\documentclass[12pt, oneside]{book}

\usepackage[utf8]{inputenc}
\usepackage{lmodern}

\usepackage{amssymb}
\usepackage{amssymb}
\usepackage{amssymb}
\usepackage{amsthm}
\usepackage{amsmath}
\usepackage[only,Yup]{stmaryrd} 
\usepackage{mathrsfs}    

\usepackage[a4paper,left=2.5cm,top=3cm,right=2.5cm,bottom=2.5cm]{geometry}

\usepackage{setspace}
\usepackage[all,cmtip]{xy}

\usepackage{enumerate}

\usepackage{multicol}

\usepackage{indentfirst}

\usepackage{hyperref}
\hypersetup{colorlinks=true,allcolors=black}  

\usepackage{pstricks}
\usepackage{graphicx}

\DeclareMathOperator{\diam}{diam}

\DeclareMathOperator{\bas}{bas}
\DeclareMathOperator{\hor}{hor}

\DeclareMathOperator{\rank}{rank}
\DeclareMathOperator{\symrank}{symrank}

\DeclareMathOperator{\vol}{vol}

\DeclareMathOperator{\id}{Id}

\DeclareMathOperator{\ric}{Ric}
\DeclareMathOperator{\ind}{ind}

\DeclareMathOperator*{\spannn}{span}
\DeclareMathOperator{\sym}{S}

\DeclareMathOperator{\abel}{Abel}

\mathchardef\ordinarycolon\mathcode`\:
\mathcode`\:=\string"8000
\begingroup \catcode`\:=\active
  \gdef:{\mathrel{\mathop\ordinarycolon}}
\endgroup

\begin{document}

\title{Introduction to Orbifolds}

\author{Francisco C.~Caramello Jr.}
\date{}

\theoremstyle{definition}
\newtheorem{example}{Example}[section]
\newtheorem{definition}[example]{Definition}
\newtheorem{remark}[example]{Remark}
\newtheorem{exercise}[example]{Exercise}
\theoremstyle{plain}
\newtheorem{proposition}[example]{Proposition}
\newtheorem{theorem}[example]{Theorem}
\newtheorem{lemma}[example]{Lemma}
\newtheorem{corollary}[example]{Corollary}
\newtheorem{claim}[example]{Claim}
\newtheorem{scholium}[example]{Scholium}
\newtheorem{thmx}{Theorem}
\renewcommand{\thethmx}{\Alph{thmx}} 
\newtheorem{corx}{Corollary}
\renewcommand{\thecorx}{\Alph{corx}} 

\newcommand{\dif}[0]{\mathrm{d}}
\newcommand{\od}[2]{\frac{d #1}{d #2}}
\newcommand{\pd}[2]{\frac{\partial #1}{\partial #2}}
\newcommand{\dcov}[2]{\frac{\nabla #1}{d #2}}
\newcommand{\proin}[2]{\left\langle #1, #2 \right\rangle}
\newcommand{\f}[0]{\mathcal{F}}
\newcommand{\g}[0]{\mathcal{G}}
\newcommand{\piorb}[0]{\pi_1^{\mathrm{orb}}}
\newcommand{\chiorb}[0]{\chi^{\mathrm{orb}}}
\newcommand{\dvol}[0]{d\!\vol}
\newcommand\blfootnote[1]{%
  \begingroup
  \renewcommand\thefootnote{}\footnote{#1}%
  \addtocounter{footnote}{-1}%
  \endgroup
}
\newcommand{\tl}[1]{\widetilde{#1}}
\newcommand{\mc}[1]{T_x\mathcal{#1}}

\begin{titlepage}
\begin{center}
\ \\
\vspace{200pt}

\textbf{\huge{Introduction to Orbifolds}}

\vspace{50pt}

\large{Francisco C. Caramello Jr.}\footnote{Address: Departamento de Matemática, Universidade Federal de Santa Catarina, R. Eng. Agr. Andrei Cristian Ferreira, 88040-900 Florianópolis -- SC, Brazil, \texttt{francisco.caramello@ufsc.br}\\
The author was supported by grant \#2018/14980-0, São Paulo Research Foundation (FAPESP).}
\vfill

\end{center}
\thispagestyle{empty}
\end{titlepage}

\pagenumbering{roman}\setcounter{page}{1}

\newpage
\thispagestyle{empty}
\begin{flushright}
\ \\
\vfill
\textit{``Manifolds are fantastic spaces. It’s a pity that there aren’t more of them.''}\\
--- nLab page on generalized smooth spaces.

\end{flushright}\ \\\ \\

\chapter*{Preface}
\addcontentsline{toc}{chapter}{Preface}

Orbifolds, first defined by I.~Satake in \cite{satake} as $V$-manifolds, are amongst the simplest generalizations of manifolds that include singularities. They are topological spaces locally modeled on quotients of $\mathbb{R}^n$ by a finite group action, and appear naturally in many areas of mathematics. Orbifolds also arise in physics as configuration spaces after one removes formal symmetries of a system, for example gauge transformations in Yang--Mills theory and coordinate transformations in general relativity \cite{emmrich}. There are many different ways to approach orbifolds, for example as Lie groupoids (see Remark \ref{remark: orbifolds as groupoids}), length spaces (see Remark \ref{Remark: orbifolds as metric spaces}), Deligne--Mumford stacks (see, e.g., \cite{lerman}), etc. Here we will adopt the more elementary, classical approach via local charts and atlases, following mostly \cite{adem}, \cite{boyer}, \cite{chenruan}, \cite{choi}, \cite{kleiner}, \cite{mrcun} and \cite{thurston}, which can be used for further reading on the subject.

Manifolds are very well-behaved spaces. While this is comfortable on the level of the objects, it forces the category of manifolds and smooth maps to have poor algebraic properties (e.g. it is not closed under limits, co-limits, quotients...), hence the pursuit of generalizations. Orbifolds arise in this context---together with the more general Chen spaces \cite{chen}, diffeological spaces \cite{iglesias}, differentiable staks \cite{behrend}, Frölicher spaces \cite{frolicher} and many others (see also \cite{stacey})---providing a category which behaves a little better at least with respect to quotients, while retaining close similarity with the realm and language of manifolds.

In Chapter \ref{section: definition and examples} we introduce the notion of orbifolds and give some examples, motivating it by considering quotients of group actions. Chapter \ref{section: fundamental group and coverings} is devoted to some algebraic topology, mainly concerning Thurston's generalization of the notion of fundamental group and covering maps to orbifolds. Chapter \ref{section: diff geometry} covers the basics of differential geometry and topology of orbifolds, such as smooth maps, bundles, integration, and some results on regular and equivariant De Rham cohomology of orbifolds. Finally, in Chapter \ref{section: riemannian orbifolds} we endow orbifolds with Riemannian metrics and see some generalizations of classical results from Riemannian geometry to the orbifold setting.

This document is an expanded version of the course notes for the mini-course ``Introduction to Orbifolds'', held on the \textit{Workshop on Submanifold Theory and Geometric Analysis} (WSTGA) at Federal University of São Carlos, Brazil, on August 05 – 09, 2019. Please send comments and corrections to \texttt{francisco.caramello@ufsc.br}.

\section*{Acknowledgments}
\addcontentsline{toc}{section}{Acknowledgments}

I am grateful to the organizing and scientific committees of WSTGA, specially Prof. Guillermo Antonio Lobos Villagra, for the invitation to lecture the aforementioned mini-course and for the efforts regarding the publication of these notes. I also thank Prof. Dirk Töben and Prof. Marcos Alexandrino for the support, encouragement and suggestions for the presentation, and Dr. Cristiano Augusto de Souza who studied orbifold theory with me at some point of our doctorates. During the preparation of these notes I was supported by grant \#2018/14980-0 of the São Paulo Research Foundation (FAPESP), for which I am also thankful.
\ \\

\noindent São Paulo - SP, Brazil, October 2019.

\begin{flushright}
Francisco C. Caramello Jr.\\
franciscocaramello.wordpress.com
\end{flushright}

\tableofcontents

\chapter{Definition and examples}\label{section: definition and examples}\pagenumbering{arabic}\setcounter{page}{6}

In this chapter we introduce the notion of orbifolds and give some examples. Since orbifolds are spaces locally modeled on quotients of actions of finite groups, it will be convenient for us to begin by recalling some basics of group actions. Furthermore, the quotient spaces of examples of the group actions that we will see here will be our first examples of orbifolds, the so called \textit{good} orbifolds, that is, orbifolds which are global quotients of properly discontinuous actions and which form a very relevant class.

\section{Group actions}\label{section: group actions}

An \textit{action} of a group $G$ on a set $X$ is a map $\mu:G\times X\to X$ with $\mu(e,x)=x$ and $\mu(g,\mu(h,x))=\mu(gh,x)$, for all $x\in X$ and $g\in G$. It is also usual to denote $\mu$ by juxtaposition, i.e. $\mu(g,x)=gx$, when there is no risk of confusion. The \textit{isotropy subgroup} of $x\in X$ is $G_x:=\{g\in G\ |\ gx=x\}$. One always has $G_{gx}=gG_xg^{-1}$. Notice that a $G$-action on $X$ is equivalent to a homomorphism $G\ni g \mapsto\mu^g\in\mathrm{Aut}(X)$, where $\mu^g:X\ni x\mapsto \mu(g,x)\in X$. The action is \textit{effective} when its kernel is trivial, which is equivalent to $\bigcap_{x\in X} G_x=\{e\}$. When $G_x=\{e\}$ for all $x$, the action is called \textit{free}. A map $\phi:X\to Y$, where $G$ acts on both $X$ and $Y$, is said to be \textit{equivariant} if $\phi(gx)=g(\phi(x))$ for all $x\in X$. In this case it is easy to check that $G_x<G_{\phi(x)}$.

\begin{example}[Rotations]\label{example: continuous rotations}
The group $\mathrm{SO}(2)$ acts effectively on $\mathbb{R}^2$ by rotations. Under the identifications $\mathrm{SO}(2)\cong\mathbb{S}^1<\mathbb{C}$ and $\mathbb{R}^2\cong\mathbb{C}$ this action is given by the usual product in $\mathbb{C}$. All points have trivial isotropy except the origin, for which $\mathbb{S}^1_0=\mathbb{S}^1$.
\end{example}

The orbit of $x$ is the subset $Gx:=\{gx\in X\ |\ g\in G\}$. Orbits give a partition of $X$, and the set of equivalence classes, the \textit{orbit space}, is denoted $X/G$. When $X$ is a topological space there is a natural topology on $X/G$, the quotient topology. Now suppose we are in the smooth category, i.e., $X=M$ is a manifold, $G$ is a Lie group and the action is a smooth map. Then, analogously, we would like to transfer the smooth structure to the quotient $M/G$ but, unfortunately, $M/G$ fails to be a manifold, unless the action is very specific (see below). We will be interested in generalizing manifolds to a class that includes more quotients $M/G$, but before we proceed, let us see some examples (see also \cite[Section 13.1]{thurston}).

\begin{example}[Torus]\label{example: torus}
Consider $M=\mathbb{R}^2$ and $G=\mathbb{Z}^2$ acting on $M$ by $(p,q)(x,y)=(x+p,y+q)$. Notice that this action is free. It is easy to see that $M/G$ is a torus.
\end{example}

\begin{example}[Mirror]
Consider the action of $\mathbb{Z}_2$ on $\mathbb{R}^3$ given by reflection along the plane $x=0$. Clearly, points with $x\neq 0$ have trivial isotropy, while points with $x=0$ have $\mathbb{Z}_2$ isotropy. The quotient is the positive hyperplane
$$\mathbb{H}^+:=\{(x,y,z)\in\mathbb{R}^3\ |\ x\geq 0\}= \mathbb{R}^3/\mathbb{Z}_2.$$
Of course, the physical interpretation is that an observer in $\mathbb{R}^3/\mathbb{Z}_2$ would perceive an ambient with a mirror reflection.
\end{example}

\begin{figure}
\centering{
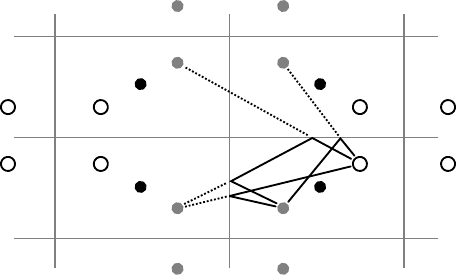}
\caption{Billiards table.}
\label{billiards}
\end{figure}

\begin{example}[Barber shop]
Again with $M=\mathbb{R}^3$, consider the group $G$ generated by reflections along the planes $x=0$ and $x=1$, i.e., $G\cong\mathbb{Z}_2*\mathbb{Z}_2$. The quotient space is $M/G=\{(x,y,z)\ |\ 0\leq x\leq1\}$. Physically this is analogous to two mirrors in parallel walls, as common in barber shops.
\end{example}

\begin{example}[Billiards table]
Consider four lines on $\mathbb{R}^2$ forming a rectangle $R$ and let $G$ be the group of isometries of $\mathbb{R}^2$ generated by reflections along those lines. Then $G$ is isomorphic to $(\mathbb{Z}_2*\mathbb{Z}_2)\times(\mathbb{Z}_2*\mathbb{Z}_2)$ and the quotient is $R$. The physical model is a billiard table (see Figure \ref{billiards}).
\end{example}

\begin{example}[Cone]\label{example: cones}
Consider the action of $\mathbb{Z}_p$ on $\mathbb{R}^2$ generated by a rotation of angle $2\pi/p$ around the origin (we also say that this is a \textit{rotation of order $p$}). The quotient is a cone with cone angle $2\pi/p$. Figure \ref{cone} shows the case $p=3$.

\begin{figure}
\centering{
\begingroup%
  \makeatletter%
  \providecommand\color[2][]{%
    \errmessage{(Inkscape) Color is used for the text in Inkscape, but the package 'color.sty' is not loaded}%
    \renewcommand\color[2][]{}%
  }%
  \providecommand\transparent[1]{%
    \errmessage{(Inkscape) Transparency is used (non-zero) for the text in Inkscape, but the package 'transparent.sty' is not loaded}%
    \renewcommand\transparent[1]{}%
  }%
  \providecommand\rotatebox[2]{#2}%
  \newcommand*\fsize{\dimexpr\f@size pt\relax}%
  \newcommand*\lineheight[1]{\fontsize{\fsize}{#1\fsize}\selectfont}%
  \ifx\svgwidth\undefined%
    \setlength{\unitlength}{226.3995115bp}%
    \ifx\svgscale\undefined%
      \relax%
    \else%
      \setlength{\unitlength}{\unitlength * \real{\svgscale}}%
    \fi%
  \else%
    \setlength{\unitlength}{\svgwidth}%
  \fi%
  \global\let\svgwidth\undefined%
  \global\let\svgscale\undefined%
  \makeatother%
  \begin{picture}(1,0.34803595)%
    \lineheight{1}%
    \setlength\tabcolsep{0pt}%
    \put(0,0){\includegraphics[width=\unitlength,page=1]{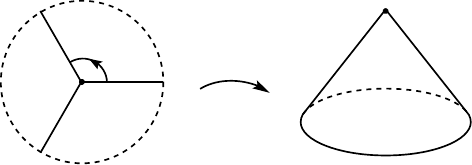}}%
    \put(0.20632742,0.25353391){\color[rgb]{0,0,0}\makebox(0,0)[lt]{\lineheight{0}\smash{\begin{tabular}[t]{l}$\frac{2\pi}{3}$\end{tabular}}}}%
  \end{picture}%
\endgroup%
}
\caption{A cone with angle $2\pi/3$.}
\label{cone}
\end{figure}
\end{example}

\begin{example}[Good football]
Similarly, on the sphere $\mathbb{S}^2$, a rotation of order $p$ around a chosen axis $R$ (say the $z$-axis) induces a $\mathbb{Z}_p$-action. The quotient space is topologically a sphere, but the poles (the points in $\mathbb{S}^2\cap R$, with non-trivial isotropy) become singular points, which locally look like the cones. In fact, the action can be lifted through the exponential map from each pole to precisely the case of Example \ref{example: cones}. We see this with more details in Example \ref{example football}, where we generalize this object.
\end{example}

\begin{example}[Pillow case]\label{example: pillow}
On $\mathbb{R}^2$, consider the discrete group $G$ generated by rotations of order $2$ around the integer lattice $\mathbb{Z}^2\subset\mathbb{R}^2$. To visualize $\mathbb{R}^2/G$, consider the square $Q=[-1/2,3/2]\times[0,1]$, so that there are two lattice points in the bottom side and other two on the top side. Then glue the vertical sides of $Q$ by the composition of the two non trivial rotations on the bottom side (or on the top side, equivalently), obtaining a cylinder. It still remains to ``zipper'' the top and bottom circles of the cylinder by the rotations, yielding an sphere with four cone points of order $2$, called a \textit{pillow case} (see Figure \ref{pillowcase}).

\begin{figure}
\centering{
\begingroup%
  \makeatletter%
  \providecommand\color[2][]{%
    \errmessage{(Inkscape) Color is used for the text in Inkscape, but the package 'color.sty' is not loaded}%
    \renewcommand\color[2][]{}%
  }%
  \providecommand\transparent[1]{%
    \errmessage{(Inkscape) Transparency is used (non-zero) for the text in Inkscape, but the package 'transparent.sty' is not loaded}%
    \renewcommand\transparent[1]{}%
  }%
  \providecommand\rotatebox[2]{#2}%
  \newcommand*\fsize{\dimexpr\f@size pt\relax}%
  \newcommand*\lineheight[1]{\fontsize{\fsize}{#1\fsize}\selectfont}%
  \ifx\svgwidth\undefined%
    \setlength{\unitlength}{331.14669817bp}%
    \ifx\svgscale\undefined%
      \relax%
    \else%
      \setlength{\unitlength}{\unitlength * \real{\svgscale}}%
    \fi%
  \else%
    \setlength{\unitlength}{\svgwidth}%
  \fi%
  \global\let\svgwidth\undefined%
  \global\let\svgscale\undefined%
  \makeatother%
  \begin{picture}(1,0.47551292)%
    \lineheight{1}%
    \setlength\tabcolsep{0pt}%
    \put(0,0){\includegraphics[width=\unitlength,page=1]{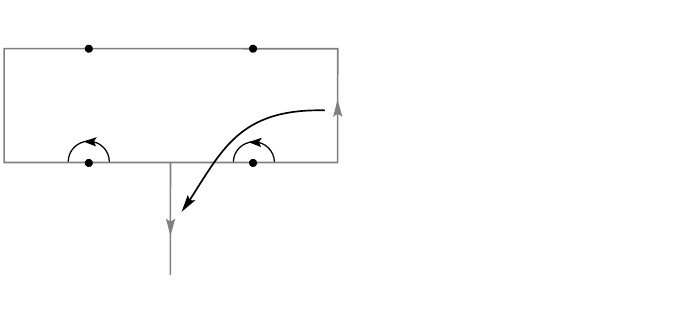}}%
    \put(0.12180882,0.20547377){\color[rgb]{0,0,0}\makebox(0,0)[lt]{\lineheight{1.25}\smash{\begin{tabular}[t]{l}$(0,0)$\end{tabular}}}}%
    \put(0,0){\includegraphics[width=\unitlength,page=2]{pillow.pdf}}%
  \end{picture}%
\endgroup%
}
\caption{The pillowcase.}
\label{pillowcase}
\end{figure}

This same space can be visualized as a quotient of the torus $\mathbb{T}^2$. Consider $\mathbb{T}^2$ as a torus of revolution in $\mathbb{R}^3$, obtained by rotating a circle in the $xz$-plane centered at $(1,0,0)$ around the $z$-axis. Then rotation by $\pi$ around the $y$-axis is a symmetry of $\mathbb{T}^2$. It is easy to see that the quotient $\mathbb{T}^2/\mathbb{Z}_2$ is again a pillow case.
\end{example}

Let us now recall some properties of smooth Lie group actions on manifolds. For every $x$, the action induces an injective immersion $G/G_x\to Gx$. Moreover, when the action is \textit{proper}, i.e., when the map
$$G\times M\ni (g,x)\longmapsto (g(x),x)\in M\times M$$
is proper (e.g., when $G$ is compact) $G/G_x\cong Gx$ is a diffeomorphism (see e.g. \cite[Proposition 3.41]{alex}). The quotient of a free proper action on a manifold is also a manifold. In fact $M\to M/G$ naturally becomes a principal $G$-bundle in this case \cite[Theorem 3.34]{alex}. A proper action by a discrete group is called \textit{properly discontinuous}. Notice that we do not assume that properly discontinuous actions are free, as it is common in some contexts. In fact, an action is properly discontinuous, as defined here, if and only if for every compact $K\subset M$ the set $\{g\in G\ |\ K\cap g K\neq\emptyset\}$ is finite. Any action of a \emph{finite} group $G$ on $M$ is automatically properly discontinuous.

We can now proceed to the abstract definition orbifolds as mathematical objects that accommodate quotients $M/G$. We will go a step further and consider then to be only \emph{locally} modeled by quotients by finite group actions, in analogy to manifolds that are locally Euclidean. In fact, orbifolds which are global quotients by properly discontinuous actions are usually called \textit{good}, and those which are quotients by finite groups are \textit{very good} (see also Section \ref{section: quotient orbifolds}).

\section{Orbifolds}

Let $X$ be a topological space and fix $n\in\mathbb{N}$. An \textit{orbifold chart} $(\widetilde{U},H,\phi)$ of dimension $n$ for an open set $U\subset X$ consists of a connected open subset $\widetilde{U}\subset\mathbb{R}^n$, a finite group $H$ acting smoothly and effectively\footnote{In some appearances of orbifolds it is more natural and useful to consider possibly non-effective actions of $H$ on $\widetilde{U}$ (see, for example, \cite{chenruan}). In spite of this we will avoid non-effective orbifolds to keep the presentation simple.} on $\widetilde{U}$ and a continuous $H$-invariant map $\phi:\widetilde{U}\to X$ that induces a homeomorphism between $\widetilde{U}/H$ and $U$ (see Figure \ref{chart}).

\begin{figure}
\centering{
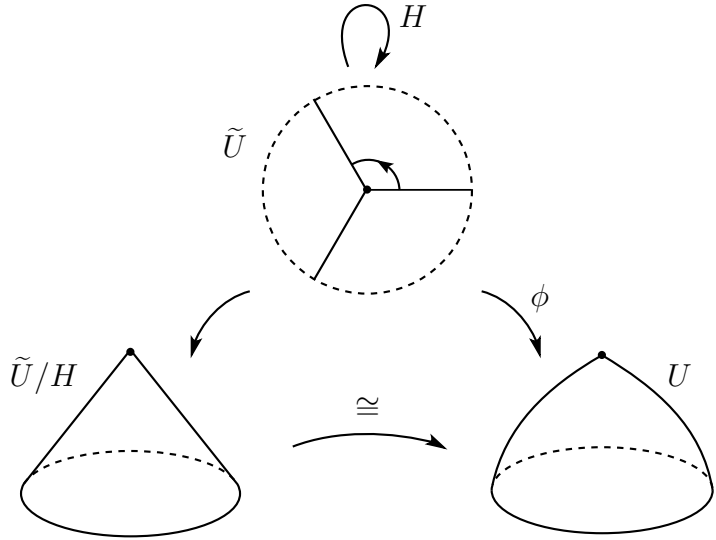}
\caption{An orbifold chart.}
\label{chart}
\end{figure}

An \textit{embedding} $\lambda:(\widetilde{U}_1,H_1,\phi_1)\hookrightarrow (\widetilde{U}_2,H_2,\phi_2)$ between two orbifold charts is a smooth embedding $\lambda:\widetilde{U}_1\hookrightarrow\widetilde{U}_2$ that satisfies $\phi_2\circ\lambda=\phi_1$. Note that for every chart $(\widetilde{U},H,\phi)$, each $h$ in the \textit{chart group} $H$ is, in particular, an embedding $(\widetilde{U},H,\phi\circ h)\hookrightarrow (\widetilde{U},H,\phi)$.

An \textit{orbifold atlas} for $X$ is a collection $\mathcal{A}=\{(\widetilde{U}_i,H_i,\phi_i)\}_{i\in I}$ of orbifold charts that covers $X$ and are locally compatible in the following sense: for any two charts $(\widetilde{U}_i,H_i,\phi_i)$, $i=1,2$, and $x\in U_1\cap U_2$, there is an open neighborhood $U_3\subset U_1\cap U_2$ containing $x$ and an orbifold chart $(\widetilde{U}_3,H_3,\phi_3)$ for $U_3$ that admits embeddings in $(\widetilde{U}_i,H_i,\phi_i)$, $i=1,2$. We say that an atlas $\mathcal{A}$ \textit{refines} an atlas $\mathcal{B}$ when every chart in $\mathcal{A}$ admits an embedding in some chart in $\mathcal{B}$. Two atlases are \textit{equivalent} if they have a common refinement. As in the manifold case, an orbifold atlas is always contained in a unique maximal one and two orbifold atlases are equivalent if, and only if, they are contained in the same maximal one.

An $n$-dimensional \textit{smooth orbifold} $\mathcal{O}$ consists of a Hausdorff paracompact topological space $|\mathcal{O}|$ together with an \textit{orbifold structure}, that is, an equivalence class $[\mathcal{A}]$ of $n$-dimensional orbifold atlases for $|\mathcal{O}|$. We will say that an orbifold chart is a \textit{chart of $\mathcal{O}$} when it is an element of some atlas in $[\mathcal{A}]$.

\begin{example}[Manifolds are orbifolds]
Observe that if the groups $H_i$ are all trivial for some atlas in $[\mathcal{A}]$, then $\mathcal{O}$ is locally Euclidean and, therefore, a manifold.
\end{example}

We will see below, in Section \ref{section: quotient orbifolds}, that quotients of almost free\footnote{Recall that a $G$-action is \textit{almost free} when $G_x$ is finite, for all $x$.} proper Lie group actions are orbifolds, in particular all examples in Section \ref{section: group actions} are orbifolds. Somewhat in the other direction, Cartesian products provide new orbifolds from old ones.

\begin{exercise}[Cartesian products]
Let $\mathcal{O}$ and $\mathcal{P}$ be smooth orbifolds. Prove that $|\mathcal{O}|\times|\mathcal{P}|$ have a natural orbifold structure given by products of charts $(\widetilde{U}\times\widetilde{V},H\times K,\phi\times\psi)$.
\end{exercise}

\begin{example}[An orbifold whose underlying space is not a topological manifold]
The quotient $\mathbb{R}^3/\mathbb{Z}_2$, where $\mathbb{Z}_2$ acts by $1\mapsto-\mathrm{id}$, is an open cone over $\mathbb{RP}^2$, hence contractible. If it where a manifold, then removing a point should not change its fundamental group. But removing the cone vertex yields a space homotopy equivalent to $\mathbb{RP}^2$, whose fundamental group is not trivial.
\end{example}

Below we list some technical consequences of the definition.

\begin{enumerate}[(i)]
\item \label{item:induced compatible charts} If $(\widetilde{U},H,\phi)$ is a chart for $U$ and $U'\subset U$ is connected, then it is easy to prove that $(\widetilde{U}',H',\phi')$ is a compatible chart for $U'$, where $\widetilde{U}'$ is a connected component of $\phi^{-1}(U')$, $H':=H_{\widetilde{U}'}$ is the subgroup that preserves $\widetilde{U}'$ and $\phi':=\phi|_{\widetilde{U}'}$.

\item Complementing item (\ref{item:induced compatible charts}), if $(\widetilde{U}_i,H_i,\phi_i)$, $i=1,2$, are compatible charts, $\widetilde{U}_1$ is simply-connected and $\phi_1(\widetilde{U}_1)\subset\phi_2(\widetilde{U}_2)$, then there is an embedding
$$\lambda:(\widetilde{U}_1,H_1,\phi_1)\lhook\joinrel\longrightarrow(\widetilde{U}_2,H_2,\phi_2).$$

\item \label{item:property two embeddings} For two embeddings $\lambda,\lambda':(\widetilde{U}_1,H_1,\phi_1)\hookrightarrow (\widetilde{U}_2,H_2,\phi_2)$ there exists a unique $h\in H_2$ such that $\lambda'=h\circ\lambda$ (the proof is not trivial, see \cite{pronk}, Proposition A.1).

\item As a consequence of item (\ref{item:property two embeddings}), any embedding $\lambda:(\widetilde{U}_1,H_1,\phi_1)\hookrightarrow (\widetilde{U}_2,H_2,\phi_2)$ induces a monomorphism $\overline{\lambda}:H_1\to H_2$. In fact, since each $g\in H_1$ is an embedding of $(\widetilde{U}_1,H_1,\phi_1)$ onto itself, for the embeddings $\lambda$ and $\lambda\circ g$ there exists a unique $h\in H_2$ with $\lambda\circ g=h\circ\lambda$. Then we can define $\overline{\lambda}(g):=h$.

\item An orbifold chart $(\widetilde{U},H,\phi)$ is a \textit{linear chart} when $H<\mathrm{O}(n)$ (acting linearly). One can always obtain an atlas consisting only of linear charts, as follows. Let $(\widetilde{U},H,\phi)$ be any chart. Since $H$ is finite, we can choose an $H$-invariant Riemannian metric on $\widetilde{U}$. Then for $\tilde{x}\in \widetilde{U}$, the exponential map gives a diffeomorphism between an open neighborhood $\widetilde{U}_{\tilde{x}}$ of the origin in $T_{\tilde{x}}\widetilde{U}$ and an $H_{\tilde{x}}$-invariant neighborhood of $\tilde{x}$. Since $H$ acts by isometries, $\exp$ is $H_{\tilde{x}}$-equivariant with respect to the action of $H_{\tilde{x}}$ on $T_{\tilde{x}}$ by differentials of its action on $\widetilde{U}$. This action is isometric, hence $H_{\tilde{x}}<\mathrm{O}(T_{\tilde{x}}\widetilde{U})$. So this gives rise to a compatible linear chart $(\widetilde{U}_{\tilde{x}},H_{\tilde{x}},\phi\circ\exp_{\tilde{x}})$ for a neighborhood of $\tilde{x}$. Doing this for each point on each chart of an atlas we obtain a compatible atlas consisting of linear charts only.
\end{enumerate}

\section{Local groups and the canonical stratification}\label{section: canonical stratification}

Let $x\in|\mathcal{O}|$ and consider a chart $(\widetilde{U},H,\phi)$ with $x=\phi(\tilde{x})\in U$. The \textit{local group} $\Gamma_x$ at $x$ is the isomorphism class\footnote{We will denote both the isomorphism class and a representative of it by $\Gamma_x$, when the meaning is clear from the context.} of the isotropy subgroup $H_{\tilde{x}} <H$. It is independent of both the chart and the lift $\tilde{x}$ (see \cite{adem}, p. 4), and for every $x\in|\mathcal{O}|$ we can always find a compatible chart $(\widetilde{U},\Gamma_x,\phi)$ \textit{around} $x$, that is, such that $\phi^{-1}(x)$ consists of a single point $\tilde{x}$. We denote by $\Sigma_\Gamma$ the subset of $|\mathcal{O}|$ of the points with local group $\Gamma$. The decomposition
$$|\mathcal{O}|=\bigsqcup_\alpha \Sigma_\alpha,$$
where each $\Sigma_\alpha$ is a connected component of some $\Sigma_\Gamma$ called a \textit{stratum}, is the \textit{canonical stratification} of $\mathcal{O}$. Each $\Sigma_\alpha$ is a manifold without boundary (see \cite{dragomir}, p. 74 ff.). The \textit{regular locus} $\Sigma_{\{e\}}$ of \textit{regular points} is an open manifold, which will also be denoted by $\mathcal{O}_\mathrm{reg}$. Since the local groups act effectivelly, $\mathcal{O}_\mathrm{reg}$ is also dense in $|\mathcal{O}|$. When $\mathcal{O}$ is connected $\mathcal{O}_\mathrm{reg}$ is also connected, hence a stratum. The subset $\mathcal{O}_{\mathrm{sing}}:=|\mathcal{O}|\setminus\mathcal{O}_\mathrm{reg}$ of \textit{singular points} is a closed subset of $|\mathcal{O}|$ with empty interior, called the \textit{singular locus} of $\mathcal{O}$. Within $\mathcal{O}_{\mathrm{sing}}$ there are two subsets which are usually useful to distinguish. The union of all strata of codimension $1$ is the \textit{mirror locus} of $\mathcal{O}$, denoted $\mathcal{O}_{\mathrm{mirr}}$. A \textit{mirror point} $x\in\mathcal{O}_{\mathrm{mirr}}$ have $\Gamma_x=\mathbb{Z}_2$, whose linearized action is by a reflection. At the other extreme, the union of all strata of minimal dimension $\mathcal{O}_{\mathrm{min}}$ is the \textit{minimal locus} of $\mathcal{O}$.

\begin{example}[Footballs and teardrops]\label{example football}
On the sphere $\mathbb{S}^2$, consider normal geodesic balls $B_i$, for $i=1,2$, centered at the north and the south poles, $N$ and $S$, respectively, such that $\mathbb{S}^2=B_1\cup B_2$. So each ball $\widetilde{B}_i:=B_R(0)\subset\mathbb{R}^2$, with $\pi/2<R<\pi$, is mapped diffeomorphically over $B_i$ by the exponential map (with respect to the usual round metric on $\mathbb{S}^2$). We use polar coordinates $(r,\theta)$, with $0\leq r< R$ and $0\leq\theta<2\pi$, on $B_R(0)$.

Let $p_i\in \mathbb{N}$ and consider the orbifold chart $(\widetilde{B}_i,\mathbb{Z}_{p_i},\phi_i)$, where $\mathbb{Z}_{p_i}$ acts on $\widetilde{B}_i$ by a rotation of order $p_i$ and $\phi_i:\widetilde{B}_i\to B_i$ maps $(r,\theta)$ to the point with geodesic coordinates $(r,p_i\theta)$. The map $\phi_i^{-1}\circ \phi_j$, on the annulus $\pi-R<r<R$, is given by
$$(r,\theta)\longmapsto \left(\pi-r,\frac{p_j}{p_i}\theta\right).$$
These are local diffeomorphisms which commute with the charts. This is sufficient to conclude that the charts are compatible (see Section \ref{subsection: associated pseudogroups}), hence we obtain an orbifold structure on $\mathbb{S}^2$, called the \textit{$(p_1,p_2)$-football}. For $p_i\neq 1$ the singular locus is, of course, $\{N,S\}$. In the special case $p_2=1$ the south pole becomes a regular point, and the resulting orbifold is called the \textit{$p_1$-teardrop}. The $(1,1)$-football is just the regular sphere.

\begin{figure}
\centering{
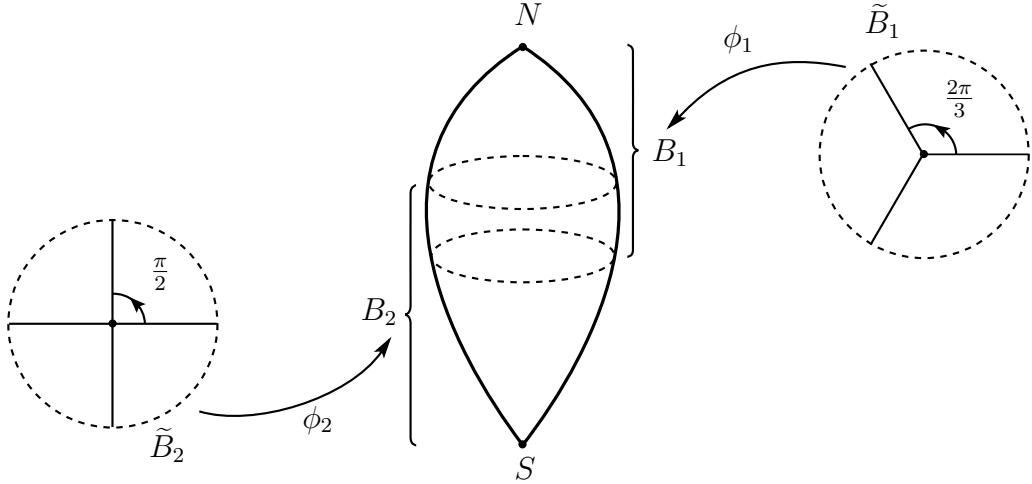}
\caption{The $(3,4)$-football}
\label{figfootball}
\end{figure}
\end{example}

We say that $\mathcal{O}$ is \textit{locally orientable} if there is an atlas $\mathcal{B}=\{(\widetilde{U}_i,H_i,\phi_i)\}\in[\mathcal{A}]$ such that each $H_i$ acts by orientation-preserving diffeomorphisms of $\widetilde{U}_i$. In particular, in this case we can suppose all charts satisfy $H_i<\mathrm{SO}(n)$. If we can choose an orientation for each $\widetilde{U}_i$ that makes every embedding between charts of $\mathcal{B}$ orientation-preserving, then $\mathcal{O}$ is \textit{orientable}. Of course, with such orientations chosen, $(\mathcal{O},\mathcal{B})$ is an \textit{oriented} orbifold.

\begin{exercise}[Singular locus of locally orientable orbifolds]
Prove that for a locally orientable orbifold all singular strata have codimension at least $2$; that is, $\mathcal{O}_{\mathrm{mirr}}=\emptyset$.
\end{exercise}

Orbifolds with boundary are defined similarly to the manifold case, by requiring the sets $\widetilde{U}$ in the charts to be open subsets of $[0,\infty)\times\mathbb{R}^{n-1}$. In order to keep the presentation simple we will avoid working with orbifolds with boundary, but we mention that the majority of the results presented here are also valid for them (even for orbifolds with corners). It is worth noting that one can have $\partial\mathcal{O}=\emptyset$ while $|\mathcal{O}|$ is homeomorphic to a topological manifold with non-empty boundary. In fact, we have the following.

\begin{example}[Silvering]\label{example: silvering}
If $M$ is a manifold with boundary we can give an orbifold structure (without boundary) $\mathcal{M}$ to $M$ so that $\partial M$ becomes a mirror. Any point $x\in\partial M$ has a neighborhood modeled on $\mathbb{R}^n/\mathbb{Z}_2$, where the action of $\mathbb{Z}_2$ is generated by reflection along the hyperplane that models $\partial M$. Then of course $\mathcal{M}_{\mathrm{sing}}=\mathcal{M}_{\mathrm{mirr}}=\partial M$, and $M=|\mathcal{M}|$ (as topological spaces). The smooth structures on $M$ and $\mathcal{M}$, however, are different (see Example \ref{example: silvering changes smooth structure}).
\end{example}

\section{Smooth maps}\label{section: smooth maps between orbifolds}

There are several different notions of smooth maps between orbifolds. They were first introduced in \cite{satake} in the most intuitive way, but it was later discovered that the concept needed some refinements in order to allow some usual constructions, for instance coherently pulling (orbi)bundles back by smooth maps. To overcome this, more subtle notions of morphisms between orbifolds were introduced. The \textit{strong maps} of \cite{pronk}, for example, match the definition of generalized groupoid homomorphisms when the orbifolds are seen as Lie groupoids (see Remark \ref{remark: orbifolds as groupoids}). There is also the notion of \textit{good map} of \cite{chenruan} (which is actually equivalent to that of strong map, see \cite{lupercio}, Proposition 5.1.7). In \cite{borzellino2} the authors present other four different notions of smooth maps between orbifolds and study their relations. Here we will follow \cite{kleiner} (whose definition coincides with that of \textit{reduced orbifold map} of \cite{borzellino2}), for it refines Satake's definition by handling the algebraic information on the singularities more carefully whilst retaining the classical differential geometric flavor (and without getting too technical\footnote{The price is that it does not solve the issue with pullbacks}).

Let $\mathcal{O}$ and $\mathcal{P}$ be orbifolds. A \textit{smooth map} $f:\mathcal{O}\to\mathcal{P}$ consists of a continuous map $|f|:|\mathcal{O}|\to|\mathcal{P}|$ which admits a \textit{smooth local lift} at each $x\in |\mathcal{O}|$, that is, there are:
\begin{enumerate}[(i)]
\item charts $(\widetilde{U},\Gamma_x,\phi)$ and $(\widetilde{V},\Gamma_{|f|(x)},\psi)$ around $x$ and $|f|(x)$, respectively, such that $|f|(U)\subset V$,
\item a homomorphism $\overline{f}_x:\Gamma_x\to \Gamma_{|f|(x)}$,
\item a smooth $\overline{f}_x$-equivariant\footnote{That is, satisfying $\widetilde{f}_x(g\tilde{y})=\overline{f}_x(g)\widetilde{f}_x(\tilde{y})$ for each $g\in\Gamma_x$ and $\tilde{y}\in\widetilde{U}$.} map $\widetilde{f}_x:\widetilde{U}\to\widetilde{V}$ such that
$$\xymatrix{
\widetilde{U} \ar[r]^{\widetilde{f}_x} \ar[d]_{\phi}& \widetilde{V} \ar[d]^{\psi}\\
U \ar[r]_{|f|} & V}$$
commutes (see Figure \ref{orbifoldmap}).
\end{enumerate}

\begin{figure}
\centering{
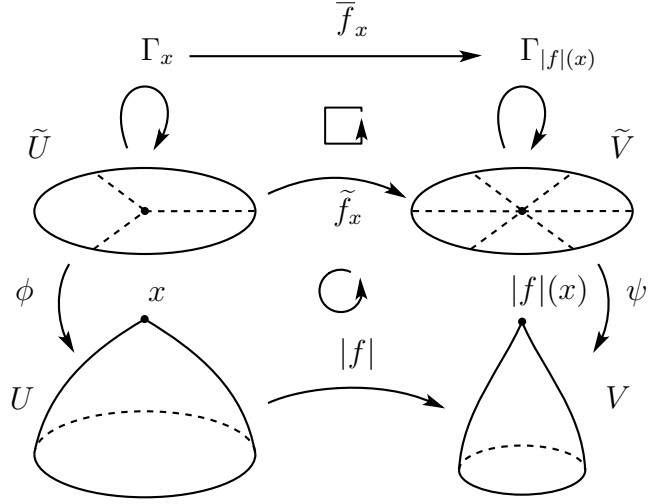}
\caption{A local lift.}
\label{orbifoldmap}
\end{figure}

A smooth map $f:\mathcal{O}\to\mathcal{P}$ is a \textit{diffeomorphism} if it admits a smooth inverse. In this case we clearly have $\Gamma_x\cong \Gamma_{|f|(x)}$ for all $x\in|\mathcal{O}|$, that is, diffeomorphisms must preserve the orbifold stratification.

\begin{exercise}
Show that the class of smooth orbifolds with smooth maps as morphisms form a category, that is, the identity map $\mathrm{id}:\mathcal{O}\to\mathcal{O}$ is smooth, for any $\mathcal{O}$, and the composition $g\circ f$ of any two smooth maps $f:\mathcal{O}\to\mathcal{P}$ and $g:\mathcal{P}\to\mathcal{Q}$ is smooth.
\end{exercise}

The space of $C^\infty$ maps $f:\mathcal{O}\to \mathbb{R}$ is denoted $C^\infty(\mathcal{O})$.

\begin{example}\label{example: silvering changes smooth structure}
Consider $M=[0,\infty)$ as a manifold with boundary and the corresponding orbifold $\mathcal{M}=\mathbb{R}/\mathbb{Z}_2$, where $\mathbb{Z}_2$ acts by reflection through $0$, obtained by silvering $\partial M$ (see Example \ref{example: silvering}). Then, by definition, $f\in C^\infty(M)$ if, and only if, it admits a smooth extension to $(-\varepsilon,\infty)$; whereas $f\in C^\infty(\mathcal{M})$ if, and only if, it admits a smooth lift\footnote{Notice that $\overline{f}$ is trivial (since the codomain is a manifold), hence equivariance becomes invariance for $\tilde{f}$.} $\tilde{f}\in C^\infty(\mathbb{R})^{\mathbb{Z}_2}$, i.e., it extends to a smooth even function. In particular, all odd derivatives $\tilde{f}^{(2k+1)}(0)$ must vanish.
\end{example}

Two different smooth local lifts at $x\in|\mathcal{O}|$ do not always differ by composition with some element of $\Gamma_{|f|(x)}$, as the following example shows.

\begin{example}[Different local lifts]\label{exemple: different local lifts}
Consider the action of $\mathbb{Z}_4$ on $\mathbb{R}\times\mathbb{C}$ generated by the multiplication by $\mathrm{i}=\sqrt{-1}$ on $\mathbb{C}$ and let $\mathcal{O}$ be the corresponding quotient orbifold. Define $\widetilde{f}_1,\widetilde{f}_2:\mathbb{R}\to\mathbb{R}\times\mathbb{C}$ by $\widetilde{f}_1(t)=(t,e^{-t^{-2}})$ and
$$\widetilde{f}_2(t)=\left\{\begin{array}{rcl}
(t,e^{-t^{-2}}) & \mbox{if} & t\leq 0,\\
(t,\mathrm{i}e^{-t^{-2}}) & \mbox{if} & t>0.\end{array}\right.$$
It is clear that $\widetilde{f}_1$ and $\widetilde{f}_2$ are local lifts for the same underlying map $|f|:\mathbb{R}\to|\mathcal{O}|$ and that they do not differ by an element of $\mathbb{Z}_4$. The more technical definitions of smooth maps between orbifolds that we mentioned above take this kind of phenomenon into account by considering these distinct lifts to represent different orbifold maps, even though they represent the same underlying map.
\end{example}

\begin{exercise}
Show that the projections $\mathcal{O}_1\times\dots\times\mathcal{O}_n\to\mathcal{O}_i$ are smooth maps.
\end{exercise}

More properties of smooth maps will be seen in Section \ref{section: diff geometry}, where we present some results from elementary differential topology of orbifolds. 

\section{Quotient orbifolds}\label{section: quotient orbifolds}

Orbifolds appear naturally as quotients of smooth Lie group actions. More precisely, we have the following result.

\begin{proposition}\label{prop: quotient orbifolds}
Suppose that a Lie group $G$ acts properly, effectively and almost freely on a smooth manifold $M$. Then the quotient space $M/G$ has a natural orbifold structure.
\end{proposition}

\begin{proof} Let us just sketch the proof. As $G\times M\ni (g,x)\mapsto (g x, x)\in M\times M$ is a proper map between locally compact spaces, it is also closed, hence $R:=\{(x,y)\in M\times M\ |\ G(x)=G(y)\}$ is closed. As the quotient projection $\pi:M\to M/G$ is an open map, it follows that $\pi((M\times M)\setminus R)$, the complement of the diagonal in $M/G\times M/G$, is open. Therefore $M/G$ is Hausdorff. Moreover, it is clearly paracompact.

Now, for any $x\in M$ there is a slice (see \cite{alex}, Theorem 3.49) $S_x=\exp^\perp(B_\varepsilon(0))$ (with respect to a suitable Riemannian metric on $M$) on which the finite isotropy subgroup $G_x$ acts. Defining $\mathrm{Tub}(Gx):=G(S_x)$, the tubular neighborhood theorem (see, for instance, \cite{alex}, Theorem 3.57) asserts that $\mathrm{Tub}(Gx)/G\cong S_x/G_x$, so $(B_\varepsilon(0),G_x,\pi\circ\exp^\perp)$ is an orbifold chart around $\pi(x)$, where we consider the linearized action of $G_x$ on $B_\varepsilon(0)$ via the isotropy representation $G_x<\mathrm{GL}(T_xS_x)$.\end{proof}

We will denote the \textit{quotient orbifold} obtained this way by $M/\!/G$ in order to differentiate it from its underlying topological space $M/G$. That is, we have
$$|M/\!/G|=M/G.$$
We say that an orbifold $\mathcal{O}$ which is diffeomorphic to a quotient orbifold $M/\!/G$ is \textit{good} or \textit{developable}, when $G$ is discrete, and \textit{very good} when $G$ is finite. Otherwise we have a \textit{bad} orbifold. 

\begin{example}
All examples from Section \ref{section: group actions}, except Example \ref{example: continuous rotations}, are good orbifolds. Notice that the same orbifold can appear as a quotient by a discrete group and also as a quotient by a finite group, such as the pillow case in Example \ref{example: pillow}.
\end{example}

\begin{exercise}
Show that the quotient space of the action in Example \ref{example: continuous rotations} can be realized as a very good orbifold.
\end{exercise}

Let us see another example coming from a non-discrete action.

\begin{example}[Weighted projective space]\label{example: weighted complex projective space}
Fix $\lambda=(\lambda_0,\dots,\lambda_n)\in\mathbb{N}^{n+1}$ satisfying $\gcd(\lambda_0,\dots,\lambda_n)=1$. We now modify the standard action of $\mathbb{C}^\times$ on $\mathbb{C}^{n+1}\setminus\{0\}$ by adding weights given by $\lambda$. Precisely, let $z\in\mathbb{C}^\times$ act by
\begin{equation}\label{eq: weighted Hopf action} z(z_0,\dots,z_n)=(z^{\lambda_0}z_0,\dots,z^{\lambda_n}z_n).\end{equation}
We call this action the \textit{weighted Hopf action}. The quotient orbifold $(\mathbb{C}^{n+1}\setminus\{0\})/\!/\mathbb{C}^\times$ is called a \textit{weighted complex projective space}, which we denote by $\mathbb{CP}^n[\lambda_0,\dots,\lambda_n]$ (or simply $\mathbb{CP}^n[\lambda]$ when the exact weights are not relevant). Weighted projective spaces play the same role in the category of orbifolds as the usual complex projective space plays in the category of smooth manifolds.  As the later, they can also be seen as algebraic varieties and, so, they exemplify how orbifolds can appear in algebraic geometry (see, for example,\cite{dolgachev}).

The weighted Hopf action restricts to an action of $\mathbb{S}^1<\mathbb{C}^\times$ on $\mathbb{S}^{2n+1}\subset\mathbb{C}^{n+1}$ with the same quotient, so that we could equivalently define $\mathbb{CP}^n[\lambda_0,\dots,\lambda_n]:=\mathbb{S}^{2n+1}/\!/\mathbb{S}^1$. Let us study the case $n=1$ to get a better grasp of this action. There are two special (exceptional) orbits of $\mathbb{S}^1$, namely $\mathbb{S}^1(1,0)$ and $\mathbb{S}^1(0,1)$, with stabilizers $\mathbb{Z}_{\lambda_0}$ and $\mathbb{Z}_{\lambda_1}$, respectively, which coincide with the singular orbits of the $\mathbb{T}^2$-action $(t_0,t_1)(z_0,z_1)=(t_0z_0,t_1z_1)$. The other $\mathbb{S}^1$-orbits are principal and contained within the regular orbits of $\mathbb{T}^2$, defining closed $\lambda_0/\lambda_1$-Kronecker foliations on them. Figure \ref{figfootball} illustrates these orbits via stereographic projection of $\mathbb{S}^3$ (notice that one of the exceptional orbits is a circle that passes through infinity, so is projected to a vertical line). With this its not difficult to observe that $\mathbb{CP}^1[\lambda_0,\lambda_1]$ is simply the $(\lambda_0,\lambda_1)$-football (see Example \ref{hopfaction}), the poles corresponding to the exceptional $\mathbb{S}^1$-orbits. Notice also that $\mathbb{CP}^n[1,\dots,1]$ is just the usual complex projective space.

\begin{figure}
\centering{
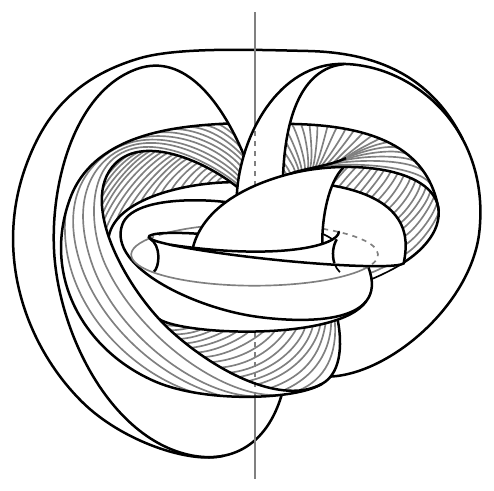}
\caption{The weighted Hopf action on $\mathbb{S}^3$.}
\label{hopfaction}
\end{figure}

Charts for $\mathbb{CP}^n[\lambda_0,\dots,\lambda_n]$ that are compatible with the induced orbifold structure as a quotient can be constructed as follows (see \cite{boyer}, Section 4.5, for further details). Cover the space $|\mathbb{CP}^n[\lambda]|$ with the $n+1$ open sets
$$U_i=\Big\{[z_0:\dots:z_n]\in\mathbb{CP}^n[\lambda]\ \Big|\ z_i\neq0\Big\},$$
where we use homogeneous coordinates as in the manifold case, that is, $[z_0:\dots:z_n]$ denotes the orbit of $(z_0,\dots,z_n)$. The charts will be $(\widetilde{U}_i,G_i,\varphi_i)$, where $\widetilde{U}_i=\mathbb{C}^n$ with affine coordinates $(w_{i0},\dots,\widehat{w}_{ii},\dots,w_{in})$ satisfying $w_{ij}^{\lambda_i}=z_j^{\lambda_i}/z_i^{\lambda_j}$, and the maps $\varphi_i:\widetilde{U}_i\to U_i$ are given by
$$\varphi_i(w_{i0},\dots,\widehat{w}_{ii},\dots,w_{in})=[w_{i0}^{\lambda_i},\dots,w_{i(i-1)}^{\lambda_i},1,w_{i(i+1)}^{\lambda_i},\dots,w_{in}^{\lambda_i}].$$
The chart groups $G_i\cong\mathbb{Z}_{\lambda_i}$ are simply the groups of $\lambda_i$th roots of the unity acting on $\widetilde{U}_i$ by multiplication.

The singular locus of $\mathbb{CP}^n[\lambda]$ consist of copies of $\mathbb{CP}^k[\lambda_k]$, $0\leq k\leq n$, that correspond to some of the coordinate $(k+1)$-dimensional subspaces of $\mathbb{C}^{n+1}$. We can visualize this stratification as an $n$-simplex, where each $k$-cell correspond to a copy of $\mathbb{CP}^k[\lambda']$. Precisely, the subset
$$\Big\{[z_0:\dots:z_n]\in\mathbb{CP}^n[\lambda_0,\dots,\lambda_n]\ \Big|\ z_j=0 \mbox{ for } j\neq i_1,\dots,i_k\Big\}$$
is singular if and only if $l:=\gcd(\lambda_{i_1},\dots,\lambda_{i_k})>1$.  In this case the local group at a generic point in this singular subset is $\mathbb{Z}_l$.
\end{example}

\chapter{Algebraic topology of orbifolds}\label{section: fundamental group and coverings}

In this chapter we will see a little of the algebraic topology of orbifolds, focusing manly on the notions of orbifold fundamental group and coverings, whose developments are mainly due to Thurston \cite{thurston}. We begin with some theory of pseudogroups, which we now recall following mostly \cite{salem}.

\section{Pseudogroups}\label{subsection: associated pseudogroups}

Let $S$ be a smooth manifold. A \textit{pseudogroup} $\mathscr{H}$ of local diffeomorphisms of $S$ consists of a set of diffeomorphisms $h:U\to V$, where $U$ and $V$ are open sets of $S$, such that

\begin{enumerate}[(i)]
\item $\id_U\in\mathscr{H}$ for any open set $U\subset S$,
\item $h\in\mathscr{H}$ implies $h^{-1}\in\mathscr{H}$,
\item if $h_1:U_1\to V_1$ and $h_2:U_2\to V_2$ are in $\mathscr{H}$, then their composition
$$h_2\circ h_1:h_1^{-1}(V_1\cap U_2)\longrightarrow h_2(V_1\cap U_2)$$
also belongs to $\mathscr{H}$,
\item if $h:U\to V$ is in $\mathscr{H}$, then its restriction $h:U'\to h(U')$, to each open set $U'\subset U$, is also in $\mathscr{H}$, and
\item if $U\subset S$ is open and $k:U\to V$ is a diffeomorphism such that $U$ admits an open cover $\{U_i\}$ with $k|_{U_i}\in\mathscr{H}$ for all $i$, then $k\in\mathscr{H}$.
\end{enumerate}

Notice that any family of local diffeomorphisms of $S$ containing the identity generates a pseudogroup, which is obtained by taking inverses, restrictions, compositions and unions of elements in the family.

The \textit{$\mathscr{H}$-orbit} of $x\in S$ consists of the points $y\in S$ for which there is some $h\in\mathscr{H}$ satisfying $h(x)=y$. The quotient by the corresponding equivalence relation, endowed with the quotient topology, is the \textit{space of orbits} of $\mathscr{H}$, that we denote $S/\mathscr{H}$.

\begin{example}[Orbifolds as pseudogroups]\label{example: orbifold associated pseudogroup} Let $\mathcal{O}$ be an orbifold and fix an atlas $\mathcal{A}=\{(\widetilde{U}_i,H_i,\phi_i)\}$. We define
$$U_\mathcal{A}:=\bigsqcup_{i\in I}\widetilde{U}_i\ \ \ \mbox{and}\ \ \ \phi:=\bigsqcup_{i\in I} \phi_i:U_\mathcal{A}\to |\mathcal{O}|,$$
that is, $x\in \widetilde{U}_i\subset U_\mathcal{A}$ implies $\phi(x)=\phi_i(x)$. A \textit{change of charts} of $\mathcal{A}$ is a diffeomorphism $h:V\to W$, with $V,W\subset U_\mathcal{A}$ open sets, such that $\phi\circ h=\phi|_V$. Note that embeddings between charts of $\mathcal{A}$ and, in particular, the elements of the chart groups $H_i$ are changes of charts. The collection of all changes of charts of $\mathcal{A}$ generates a pseudogroup $\mathscr{H}_{\mathcal{A}}$ of local diffeomorphisms of $U_\mathcal{A}$, and $\phi$ induces a homeomorphism $U_\mathcal{A}/\mathscr{H}_{\mathcal{A}}\to|\mathcal{O}|$.
\end{example}

Let $\mathscr{H}$ and $\mathscr{K}$ be pseudogroups of local diffeomorphisms of $S$ and $T$, respectively. A (smooth) \textit{equivalence} between $\mathscr{H}$ and $\mathscr{K}$ is a maximal collection $\Phi$ of diffeomorphisms from open sets of $S$ to open sets of $T$ such that $\{\mathrm{Dom}(\varphi)\ |\ \varphi\in\Phi\}$ covers $S$, $\{\mathrm{Im}(\varphi)\ |\ \varphi\in\Phi\}$ covers $T$ and, for all $\varphi,\psi\in\Phi$, $h\in\mathscr{H}$ and $k\in\mathscr{K}$, we have $\psi^{-1}\circ k\circ\varphi\in\mathscr{H}$, $\psi\circ h\circ\varphi^{-1}\in\mathscr{K}$ and $k\circ\varphi\circ h\in\Phi$, whenever these compositions make sense.

\begin{example}
Let $\mathcal{A}_i$, $i=1,2$, be two equivalent atlases for an orbifold $\mathcal{O}$, with common refinement $\mathcal{B}$. Then for each chart $(\widetilde{V}_k,H_k,\phi_k)$ of $\mathcal{B}$ we have an embedding $\lambda^k_i$ in some chart of $\mathcal{A}_i$. The maximal collection of diffeomorphisms generated by $\{\lambda^k_2\circ(\lambda^k_1)^{-1}\}$ is an equivalence between $\mathscr{H}_{\mathcal{A}_1}$ and $\mathscr{H}_{\mathcal{A}_2}$. Thus, up to equivalence, we can define the \textit{pseudogroup $\mathscr{H}_{\mathcal{O}}$ of changes of charts of $\mathcal{O}$}. We will sometimes abuse the notation and use $(U_\mathcal{O},\mathscr{H}_{\mathcal{O}})$ to mean one of its representatives, when it is clear that a different choice would not affect the results.
\end{example}

Changes of charts can be used as an alternative notion of compatibility between the charts in an orbifold atlas, yielding therefore yet another definition for orbifolds which is equivalent to our definition (see \cite{mrcun}, Proposition 2.13, for details): one could define an orbifold $\mathcal{O}$ as an equivalence class $[(\mathscr{H},S)]$ of pseudogroups of local diffeomorphisms such that $S/\mathscr{H}$ is Hausdorff and, for any $x\in S$, there is a neighborhood $U\ni x$ such that $\mathscr{H}|_U$ is generated by a finite group of diffeomorphisms of $U$.

\begin{remark}[Orbifolds as groupoids]\label{remark: orbifolds as groupoids}
The pseudogroup $\mathscr{H}_{\mathcal{O}}$ is also relevant in enabling one to associate to $\mathcal{O}$ a Lie groupoid. This is in fact the modern approach to orbifolds, and the groupoid language can provide new insights to the theory. The groupoid is obtained by simply passing to the germs of the maps in a pseudogroup representing $\mathcal{O}$: if $\mathscr{H}_{\mathcal{A}}\in \mathscr{H}_{\mathcal{O}}$, consider $\mathrm{G}_{\mathcal{A}}$ the groupoid of germs of elements in $\mathscr{H}_{\mathcal{A}}$. Then $\mathrm{G}_{\mathcal{A}}$ is a proper, effective, \textit{étale} Lie groupoid, and for a different compatible atlas $\mathcal{B}$, the groupoid $\mathrm{G}_{\mathcal{B}}$ is Morita equivalent to $\mathrm{G}_{\mathcal{A}}$ (see \cite{mrcun}, Proposition 5.29). Hence we can associate to $\mathcal{O}$ a unique Morita equivalence class of proper Lie groupoids $\mathrm{G}_{\mathcal{O}}$. Conversely, any proper, effective, \textit{étale} Lie groupoid $G_1\rightrightarrows G_0$ defines an orbifold structure on its coarse moduli space $G_0/G_1$, with $\mathrm{G}_{G_0/G_1}$ Morita equivalent to $G_1\rightrightarrows G_0$ (see \cite{mrcun}, Corollary 5.31). We refer to \cite{adem}, \cite{boyer}, \cite{lerman}, \cite{moerdijk} and \cite{pronk} for more details on orbifold theory via Lie groupoids.
\end{remark}

\section{Orbifold fundamental group}\label{section: piorb}

The notion of fundamental group can be generalized to orbifolds as homotopy classes of loops on pseudogroups representing them. This algebraic invariant will be richer then the ordinary fundamental group since it will capture some information on the singularities, besides the topological information of the underlying topological space. The construction actually works for general pseudogroups and, although it is similar to the classical one, a more elaborate notion of homotopy classes is needed in order to manage the local nature of pseudogroups. In this section we follow the presentation in \cite{salem}.

Let $\mathscr{H}$ be a pseudogroup of local diffeomorphisms  of $S$. An \textit{$\mathscr{H}$-loop} with base point $x\in S$ consists of
\begin{enumerate}[(i)]
\item a sequence $0=t_0<\dots<t_n=1$,
\item a continuous path $c_i:[t_{i-1},t_i]\to S$, for each $1\leq i\leq n$,
\item an element $h_i\in\mathscr{H}$ defined in a neighborhood of $c_i(t_i)$, for each $i$, such that $h_ic_i(t_i)=c_{i+1}(t_i)$, for $1\leq i\leq n-1$, and  $c_1(0)=h_nc_n(1)=x$ (see Figure \ref{hloop}).
\end{enumerate}

\begin{figure}
\centering{
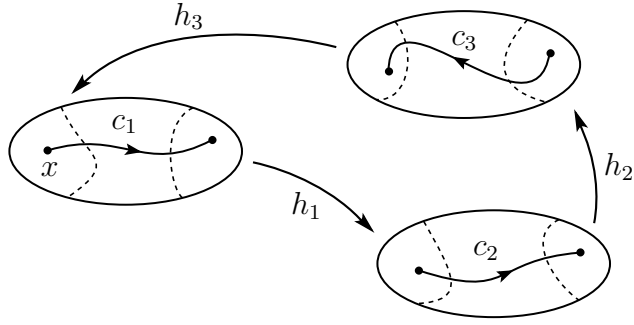}
\caption{An $\mathcal{H}$-loop, with $n=3$.}
\label{hloop}
\end{figure}

A \textit{subdivision} of a $\mathscr{H}$-loop is a new $\mathscr{H}$-loop obtained by adding points to $[0,1]$, taking restrictions of the paths $c_i$ to the new intervals formed, and taking $h=\mathrm{id}$ at the new points. Two $\mathscr{H}$-loops at $x$ are \textit{equivalent} if they admit subdivisions $(h_i,c_i)$ and $(h_i',c_i')$ such that for each $1\leq i\leq n$ there exists an element $g_i\in\mathscr{H}$ defined on a neighborhood of $c_i$ satisfying
\begin{enumerate}[(i)]
\item $g_1=\mathrm{id}$ and $g_i\circ c_i=c_i'$,
\item $h_i'\circ g_i$ and $g_{i+1}\circ h_i$ have the same germ at $c_i(t_i)$, for $1\leq i\leq n-1$,
\item $h_n'\circ g_n$ and $h_n$ have the same germ at $c_n(1)$.
\end{enumerate}
A \textit{deformation} of an $\mathscr{H}$-loop $(h_i,c_i)$ consists of deformations $c_i^s$ of the paths $c_i$ so that $(h_i,c_i^s)$ is an $\mathscr{H}$-loop based at $x$ for each $s\in[0,1]$.

Two $\mathscr{H}$-loops are in the same homotopy class if one can be obtained from the other by a finite number of subdivisions, equivalences and deformations (see Figure \ref{homotopic}). The set of homotopy classes of $\mathscr{H}$-loops based at $x\in\mathscr{H}$ forms a group $\pi_1(\mathscr{H},x)$, with the product defined by concatenation of loops. If $S/\mathscr{H}$ is connected then for any $x,y\in S$ there is an isomorphism $\pi_1(\mathscr{H},x)\cong\pi_1(\mathscr{H},y)$. In this case we will often omit the base point, when it is not relevant, denoting just $\pi_1(\mathscr{H})$. Moreover, an equivalence $\Phi$ between $(\mathscr{H},S)$ and $(\mathscr{K},T)$ clearly defines an isomorphism $\pi_1(\mathscr{H},x)\cong\pi_1(\mathscr{H},\phi(x))$, for $\phi\in\Phi$ defined in a neighborhood of $x$.

\begin{figure}
\centering{
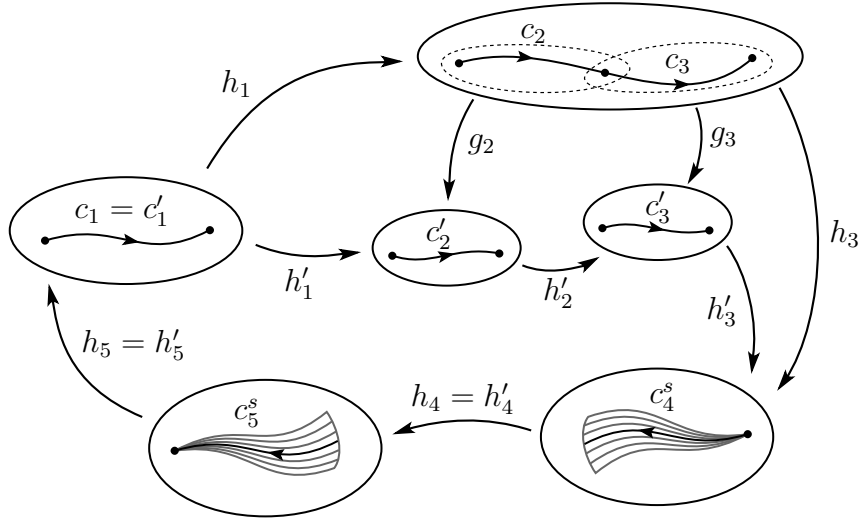}
\caption{Homotopic $\mathcal{H}$-loops.}
\label{homotopic}
\end{figure}

\begin{exercise}[Fundamental groups of manifolds]
Show that if $M$ is a connected manifold and $\mathscr{H}_M$ is its pseudogroup of changes of charts, then $\pi_1(\mathscr{H}_M)\cong\pi_1(M)$.
\end{exercise}

We define the \textit{fundamental group} of an orbifold $\mathcal{O}$ at $x\in|\mathcal{O}|$ as $\piorb(\mathcal{O},x):=\pi_1(\mathscr{H}_{\mathcal{O}},\tilde{x})$, for $\tilde{x}\in U_{\mathcal{O}}$ with $\phi(\tilde{x})=x$. From what we saw above, its isomorphism class does not depend upon the atlas we choose (nor on the lift $\tilde{x}$, in particular). By construction, there is a surjective homomorphism $\piorb(\mathcal{O})\to \pi_1(|\mathcal{O}|)$ (see \cite{haefliger4}). It is instructive to exemplify with the simple case of finite group actions.

\begin{example}[Cone]
Let $\mathcal{O}$ be the quotient orbifold of a rotation $r$ of order $p$ on $\mathbb{R}^2$, as in Example \ref{example: cones}. Then one can readily check that $(r^k,c_0)$, where $c_0$ is the constant path $c_0(t)\equiv 0$ and $0<k<p$, is not trivial, because of the condition on the germs in the definition of equivalence of loops. In fact, this argument shows that each $k$ defines a distinct homotopy class. On the other hand, any non-constant $\mathscr{H}_{\mathcal{O}}$-loop $(r^k,c)$ based at $0$ clearly can be homotoped to $(r^k,c_0)$, so we actually have $\piorb(\mathcal{O},0)\cong\mathbb{Z}_p$.

We get a more geometric interpretation of the isomorphism $\piorb(\mathcal{O})\cong\mathbb{Z}_p$ if we consider a regular base point $x\in|\mathcal{O}|$. Let $\gamma$ be an ordinary loop on $|\mathcal{O}|$ based at $x$, which winds around the cone vertex $k$-times. Notice that $\gamma$ lifts to a path $\tilde{\gamma}$ in $\mathbb{R}^2$ such that $\tilde{\gamma}(1)=r^k(\tilde{\gamma}(0))$, hence $(r^{p-k},\tilde{\gamma})$ is a $\mathscr{H}_{\mathcal{O}}$-loop. If $k\mod p\neq0$ this loop is not trivial, since any homotopy to the constant map would have to move the end point by a deformation. So, intuitively, a singularity of order $p$ is ``perceived'' by $\piorb$ as $1/p$-th of a point, in the sense that a loop in the underlying topological space has to wind $p$ times around it before it can be homotoped to the trivial path. 
\end{example}

In fact, for any orbifold chart $(\widetilde{U},H,\varphi)$ one has $\piorb(\widetilde{U}/\!/H)\cong H$. The fundamental group of more complicated orbifolds can be calculated by the orbifold version of the Seifert--Van Kampen theorem that follows. In what follows we use the notation of Example \ref{example: open suborbifolds} and, for an inclusion $i_{\mathcal{U}\mathcal{O}}$ of an open suborbifold $\mathcal{U}=\mathcal{O}|_U$, we denote the induced homomorphism by $i^*_{\mathcal{U}\mathcal{O}}:\piorb(\mathcal{U},x)\to\piorb(\mathcal{O},x)$.

\begin{theorem}[Seifert--Van Kampen theorem for orbifolds {\cite[Theorem 4.7.1]{choi}}]
Let $\mathcal{O}$ be a connected orbifold and suppose $U$ and $V$ are open sets such that $|\mathcal{O}|=U\cup V$ and $W:=U\cap V$ is connected. Then for any $x\in W$,
$$\piorb(\mathcal{O},x)\cong \piorb(\mathcal{U},x)*_{\piorb(\mathcal{W},x)}\piorb(\mathcal{V},x),$$
where the later is the amalgamated free product, that is, the quotient of $\piorb(\mathcal{U},x)*\piorb(\mathcal{V},x)$ by the normal subgroup generated by $\{i^*_{\mathcal{W}\mathcal{U}}(\gamma)i^*_{\mathcal{W}\mathcal{V}}(\gamma^{-1})\ |\ \gamma\in\piorb(\mathcal{W},x)\}$.
\end{theorem}

\begin{example}[Fundamental group of teardrops]\label{fundamental groups of teardrops}
Let $\mathcal{O}$ be the $p$-teardrop orbifold over $\mathbb{S}^2=B_1\cup B_2$ from Example \ref{example football}. Then $\mathcal{B}_1$ is a cone, hence $\piorb(\mathcal{B}_1)\cong\mathbb{Z}_p$ and $\mathcal{B}_2$ is a half-hemisphere, hence simply-connected. Moreover, $\mathcal{B}_1\cap\mathcal{B}_2$ is an annulus, thus the map induced by its inclusion in $\mathcal{B}_1$ is surjective. The Seifert--Van Kampen theorem then asserts that $\piorb(\mathcal{O})$ is trivial.
\end{example}

\begin{exercise}
Calculate the orbifold fundamental group of the $(p,q)$-football.
\end{exercise}

\section{Orbifold coverings}

A \textit{covering} of a smooth orbifold $\mathcal{O}$ is a pair $(\widehat{\mathcal{O}},\rho)$, where $\widehat{\mathcal{O}}$ is another orbifold and $\rho$ is a surjective smooth map $\widehat{\mathcal{O}}\to\mathcal{O}$ satisfying that
\begin{enumerate}[(i)]
\item For each $x\in|\mathcal{O}|$, there is a chart $(\widetilde{U},H,\varphi)$ over $x$ such that $|\rho|^{-1}(U)$ is a disjoint union of open subsets $V_i$,
\item Each $V_i$ admits an orbifold chart of the type $(\widetilde{U},H_i,\varphi_i)$, where $H_i<H$, such that $\rho$ lifts to the identity $\widetilde{\rho}_i:\widetilde{U}\to\widetilde{U}$, with $\overline{\rho}_i:H_i\hookrightarrow H$.
\end{enumerate}
When the preimage of a regular point by $|\rho|$ has cardinality $r<\infty$ we say that $\rho$ is \textit{$r$-sheeted}. Notice that, in general, $|\rho|$ is not a covering between the underlying topological spaces.

\begin{example}[Manifold coverings]
Every covering of manifolds $\widehat{M}\to M$ is an orbifold covering when the manifolds are seen as orbifolds.
\end{example}

\begin{example}[Properly discontinuous actions]
If $G$ acts properly discontinuously on a manifold $M$, then $M$ is a covering space for $M/\!/G$. In fact, $M/\!/H$ is a covering space for $M/\!/G$, for each subgroup $H<G$. Thus, a cone of order $p$, for example, covers a cone of order $kp$, for each $k\in\mathbb{N}$.
\end{example}

\begin{example}[Doubling mirrors]\label{example: doubling mirrors}
Let $\mathcal{O}$ be an orbifold with mirror singular points and let $\{(\widetilde{U}_i,H_i,\phi_i)\}$ be an atlas. For each $i$, define a new chart $(\widetilde{U}_i\times\{-1,1\},H_i,\phi_i')$, where
\begin{enumerate}[(i)]
\item $H_i$ acts by $h(\tilde{x},k)=(h\tilde{x},\mathrm{sgn}(h)k)$ where $\mathrm{sgn}(h)=1$ if the action of $h$ preserves orientation and $-1$ otherwise,
\item $\phi_i':\widetilde{U}_i\times\{-1,1\}\to\widetilde{U}_i\times\{-1,1\}/H_i=:U_i'$ is the quotient map.
\end{enumerate}
Each embedding $(\widetilde{V},G,\psi)\hookrightarrow(\widetilde{U}_i,H_i,\phi_i)$ has an obvious lift
$$(\widetilde{V}\times\{-1,1\},G,\psi')\lhook\joinrel\longrightarrow(\widetilde{U}_i\times\{-1,1\},H_i,\phi_i')$$
which define gluing maps $V'\to U_i'$. The quotient of the disjoint union $\bigsqcup_{i\in I}U_i'$ by the equivalence relation defined by these gluing maps is a Hausdorff, second countable space and hence define a locally oriented orbifold $\widehat{\mathcal{O}}$ with atlas $\{(\widetilde{U}_i\times\{-1,1\},H_i,\phi_i')\}$, which covers $\mathcal{O}$.
\end{example}

A \textit{base point} of a covering $\rho:\widehat{\mathcal{O}}\to\mathcal{O}$ is a regular point $\hat{x}\in|\widehat{\mathcal{O}}|$ that is mapped to a regular point in $|\mathcal{O}|$. A \textit{universal covering} of $\mathcal{O}$ is a covering $\rho:\widehat{\mathcal{O}}\to\mathcal{O}$ such that, given any other covering $\rho':\widehat{\mathcal{O}}'\to\mathcal{O}$ and base points $\hat{x}\in|\widehat{\mathcal{O}}|$ and $\hat{x}'\in|\widehat{\mathcal{O}}'|$ that map to the same point $x\in|\mathcal{O}|$, there exists a unique covering $\pi:\widehat{\mathcal{O}}\to\widehat{\mathcal{O}}'$ such that $\rho=\rho'\circ\pi$ and $|\pi|(\hat{x})=\hat{x}'$. For standard coverings of manifolds it is possible to show that universal coverings exist by combining all coverings of a given manifold through a fiber-product construction. Thurston refined this fiber-product construction in \cite{thurston} and adapted this idea to show that universal orbifold coverings always exist.

\begin{theorem}[Existence of universal coverings {\cite[Proposition 13.2.4]{thurston}}]
Any connected orbifold $\mathcal{O}$ admits a connected universal covering $\rho:\widehat{\mathcal{O}}\to\mathcal{O}$.
\end{theorem}

Given two coverings $\rho_i:\widehat{\mathcal{O}}_i\to\mathcal{O}$, a covering morphism is a smooth map $f:\widehat{\mathcal{O}}_1\to\widehat{\mathcal{O}}_2$ so that $\rho_1\circ f=\rho_2$. Similarly, the \textit{automorphism group} $\mathrm{Aut}(\rho)$ of a covering $\rho:\widehat{\mathcal{O}}\to\mathcal{O}$, is the group of \textit{deck transformations} of $\rho$, that is, diffeomorphisms $f:\widehat{\mathcal{O}}\to\widehat{\mathcal{O}}$ such that $\rho\circ f=\rho$.

The universal covering $\rho:\widehat{\mathcal{O}}\to\mathcal{O}$ is unique up to covering isomorphisms, and it is a Galois covering, i.e., $\mathrm{Aut}(\rho)$ acts transitively on the fibers $|\rho|^{-1}(x)$, for each $x\in|\mathcal{O}|$. Below we list more properties of universal coverings.

\begin{proposition}[Properties of the universal covering {\cite[Proposition 4.6.4]{choi}}, {\cite[Corollary 3.19]{bridson}}]
Let $\mathcal{O}$ be a connected orbifold and $\rho:\widehat{\mathcal{O}}\to\mathcal{O}$ its universal covering. Then
\begin{enumerate}[(i)]
\item \label{piorb e deck transformations} $\mathrm{Aut}(\rho)$ is isomorphic to $\piorb(\mathcal{O})$,
\item \label{good orbifolds iff manifold covering} $\rho$ induces a diffeomorphism $\widehat{\mathcal{O}}/\!/\mathrm{Aut}(\rho)\cong\mathcal{O}$,
\item For any subgroup $\Lambda<\mathrm{Aut}(\rho)$, $\rho$ induces a covering $\widehat{\mathcal{O}}/\!/\Lambda\to\mathcal{O}$ and, conversely, any covering of $\mathcal{O}$ is of this type.
\item The set of isomorphism classes of coverings of $\mathcal{O}$ is in correspondence with the set of conjugacy classes of subgroups of $\piorb(\mathcal{O})$.
\end{enumerate}
\end{proposition}

From item (\ref{good orbifolds iff manifold covering}) we see that an orbifold is good if and only if it admits a covering by a manifold. The proof of item (\ref{piorb e deck transformations}) is similar to the classical case: given  $x\in|\mathcal{O}|$ and $\tilde{x}\in|\rho|^{-1}(x)$, a deck transformation $f$ sends $\tilde{x}$ to another point $\tilde{y}\in|\rho|^{-1}(x)$. Choose an $\mathscr{H}_{\widehat{\mathcal{O}}}$-loop $(h_i,c_i)$ joining $\tilde{y}$ to $\tilde{x}$. The isomorphism $\mathrm{Aut}(\rho)\to\piorb(\mathcal{O},x)$ is given by $f\mapsto\rho_*[(h_i,c_i)]$.

\begin{exercise}
Show that $(p,q)$-footballs are bad orbifolds unless $p=q$.
\end{exercise}

\section{Triangulations and Euler characteristic}

Let $\mathcal{O}$ be a smooth orbifold. A \textit{triangulation} of $\mathcal{O}$ is a triangulation of its underlying topological space $|\mathcal{O}|$ in the usual sense, that is, a homeomorphism between a simplicial complex $T$ and $|\mathcal{O}|$. Recall the canonical stratification $|\mathcal{O}|=\bigsqcup_\alpha \Sigma_\alpha$ of from Section \ref{section: canonical stratification}. We will say that a triangulation of $\mathcal{O}$ is \textit{compatible} when it is compatible with the canonical stratification, in the sense that the interior of each cell of the triangulation (i.e., the image on $|\mathcal{O}|$ of an open face in $T$) is contained in a single stratum.

\begin{theorem}[{\cite[Theorem 4.5.4]{choi}}]
Every smooth orbifold $\mathcal{O}$ admits a compatible triangulation $T$.
\end{theorem}

For each cell $\tau$ of $T$, denote $n_\tau=|\Gamma_x|$ for some (hence any) $x$ in the interior of $\tau$. When $\mathcal{O}$ is compact, $T$ can be taken to be finite, and so the \textit{orbifold Euler characteristic of $\mathcal{O}$}
$$\chiorb(\mathcal{O}):=\sum_{\tau\in T}\frac{(-1)^{\dim\tau}}{n_\tau},$$
is well defined. Notice that $\chiorb(\mathcal{O})$ is, in general, a \emph{rational} number. Given the compatibility of the triangulation with the stratification, this can also be written in the more invariant form
$$\chiorb(\mathcal{O})=\sum_\Gamma \frac{\chi(\Sigma_\Gamma)}{|\Gamma|},$$
where the Euler characteristic $\chi(\Sigma_\Gamma)$ of the non-closed manifold $\Sigma_\Gamma$ is understood to be $\chi(\breve{\Sigma}_\Gamma)-1$, for $\breve{\Sigma}_\Gamma$ the one-point compactification of $\Sigma_\Gamma$.

\begin{proposition}
Let $\rho:\widehat{\mathcal{O}}\to\mathcal{O}$ be an $r$-sheeted orbifold covering. Then $\chiorb(\widehat{\mathcal{O}})=r\chiorb(\mathcal{O})$.
\end{proposition}

\begin{proof}
By using charts one verifies that, for each $x\in\Sigma_\alpha\subset\Sigma_\Gamma$, the inverse image $|\rho|^{-1}(x)$ is a set $\{\tilde{x}_1,\dots,\tilde{x}_\ell\}$ with each $\tilde{x}_i$ in a respective stratum $\Sigma_i\subset\Sigma_{\Lambda_i}$ of $\widehat{\mathcal{O}}$ so that
$$r=\sum_{i=1}^\ell\frac{|\Gamma|}{|\Lambda_i|}.$$
Therefore, the inverse image of a cell $\tau$ with local group $\Gamma(\tau)$ is a union of cells $\tau_1,\dots,\tau_\ell$ such that
$$\frac{r}{|\Gamma(\tau)|}=\sum_{i=1}^\ell\frac{1}{|\Lambda(\tau_i)|}.$$
Passing to the alternating sum over all cells we get the result.
\end{proof}

\begin{exercise}
Prove that for a compact, orientable $2$-orbifold $\mathcal{O}$ with cone points of order $p_i$ one has
$$\chiorb(O)=\chi(|\mathcal{O}|)-\sum_{i}\left(1-\frac{1}{p_i}\right).$$
\end{exercise}

\section{Higher homotopy groups and (co)homology}\label{section: orbifold cohomology and homotopy}

As we saw above, although the definition of $\piorb(\mathcal{O})$ via homotopies of $\mathscr{H}_{\mathcal{O}}$-loops is intuitive, it is difficult to adapt it to higher order homotopy groups. There is a way around this, if we are willing to sacrifice some of that intuition, which we will briefly comment here following \cite[Section 4.3]{boyer}. This approach consists of defining those groups via the Borel construction of the orthonormal frame bundle $\mathcal{O}^\Yup$, which will also lead us to a notion of orbifold homology and cohomology groups that are sensitive to the algebraic information contained in the orbifold structure of $\mathcal{O}$ (in contrast to De Rham cohomology, as Theorem \ref{theorem: Satake} shows). In what follows we will need to choose a Riemannian metric on $\mathcal{O}$ in order to consider $\mathcal{O}^\Yup$. We define these concepts precisely in Chapter \ref{section: riemannian orbifolds} but we have preferred to include the present section here since this chapter is dedicated to algebraic topology of orbifolds\footnote{In \cite{haefliger3} there is an equivalent construction of $B\mathcal{O}$ that does not rely on a Riemannian metric.}.

Choose a Riemannian metric on $\mathcal{O}$ (see Section \ref{section:riemannian metrics on orbifolds}) and let $E\mathrm{O}(n)\to B\mathrm{O}(n)$ be the universal principal $\mathrm{O}(n)$-bundle (see, for instance, \cite[Section 4.11]{husemoller}). The \textit{classifying space} of $\mathcal{O}$ is the space
$$B\mathcal{O}:=\mathcal{O}^\Yup\times_{\mathrm{O}(n)} E\mathrm{O}(n).$$

If we consider $\mathcal{O}$ as a Lie groupoid $\mathrm{G}_\mathcal{O}$ (see Remark \ref{remark: orbifolds as groupoids}), we can write simply $\mathcal{O}^\Yup=U_{\mathcal{O}}^\Yup/\mathrm{G}_{\mathcal{O}}$, where $U_{\mathcal{O}}$, as in Example \ref{example: orbifold associated pseudogroup}, coincides with the space of objects of $\mathrm{G}_\mathcal{O}$. Define
$$E\mathcal{O}:=U_{\mathcal{O}}^\Yup \times_{\mathrm{O}(n)} E\mathrm{O}(n).$$
The $\mathrm{G}_{\mathcal{O}}$-action on $U_{\mathcal{O}}^\Yup$ commutes with the action of $\mathrm{O}(n)$, so $\mathrm{G}_{\mathcal{O}}$ acts on both the total space $E\mathcal{O}$ (the action on the factor $E\mathrm{O}(n)$ being trivial) and the base space of the fiber bundle $E\mathcal{O}\to U_{\mathcal{O}}$. Upon quotienting by these actions one obtains a commutative diagram
$$\xymatrix{
E\mathcal{O} \ar[r] \ar[d] & B\mathcal{O} \ar[d]^p\\
U_{\mathcal{O}} \ar[r] & |\mathcal{O}|.
}$$

A fiber of $p$ over a regular point of $|\mathcal{O}|$ is $E\mathrm{O}(n)$, while a fiber over a singular point $x$ is an Eilenberg--MacLane space $K(\Gamma_x,1)$, that is, a connected topological space satisfying $\pi_1(K(\Gamma_x,1))\cong \Gamma_x$ and $\pi_i(K(\Gamma_x,1))=0$ for all $i> 1$.

\begin{example}[{\cite[Exemple 4.2.5]{haefliger3}}]
Let $\mathcal{O}$ be a $2$-dimensional orbifold, with $|\mathcal{O}|$ homeomorphic to a compact, oriented surface $S$ of genus $g$, such that $\mathcal{O}_{\mathrm{sing}}$ consists of cone points $x_1,\dots x_k$, that is, $\Gamma_{x_i}\cong \mathbb{Z}_{n_i}$, acting by rotations of order $n_i$ on a disk. The homotopy type of $B\mathcal{O}$ can be described as follows: to the wedge sum $\bigvee_i B\mathbb{Z}_{n_i}$, we attach $S\setminus B$, where $B$ is a small open ball, in such a way that the gluing map $f$ sends the boundary $\partial(S\setminus B)$ to the homotopy class of the sum of the generators of $\pi_1(B\mathbb{Z}_{n_i})$. That is,
$$B\mathcal{O}\simeq S \sqcup_f \bigvee_i B\mathbb{Z}_{n_i}.$$
\end{example}

The \textit{orbifold homology, cohomology and homotopy groups} of $\mathcal{O}$ are, by definition, the homology, cohomology and homotopy groups of $B\mathcal{O}$, that we denote, respectively, by $H_i^{\mathrm{orb}}(\mathcal{O},R)$, $H^i_{\mathrm{orb}}(\mathcal{O},R)$ and $\pi_i^{\mathrm{orb}}(\mathcal{O})$, for some unital ring $R$. This definition of $\pi_1^{\mathrm{orb}}(\mathcal{O})$ is equivalent to the one in Section \ref{section: piorb} via homotopy of $\mathscr{H}_{\mathcal{O}}$-loops (see \cite{haefliger4}). In \cite{moerdijk3} the authors obtain a simplicial version of orbifold cohomology which agrees with the above definition. That is, for each $\mathcal{O}$ they construct a simplicial set whose cohomology is isomorphic to $H^*_{\mathrm{orb}}(\mathcal{O},R)$ (even with local coefficients). A diffeomorphism $f:\mathcal{O}\to\mathcal{P}$ induces a diffeomorphism $B\mathcal{O}\to B\mathcal{P}$ (see \cite[Proposition 4.3.10]{boyer}), so these groups are invariants of the orbifold structure. Although the orbifold homology and cohomology groups are in fact sensitive to the orbifold structure, the study of the Leray spectral sequence of $p$ leads to the following result which shows that this information is contained in the torsion of these groups.

\begin{theorem}[{\cite[Proposition 4.2.3]{haefliger3}}{\cite[Corollary 4.3.8]{boyer}}]\label{theorem: cohomology of orbifolds over a field}
If $R$ is a field, then we have $H_i^{\mathrm{orb}}(\mathcal{O},R)\cong H_i(|\mathcal{O}|,R)$ and $H^i_{\mathrm{orb}}(\mathcal{O},R)\cong H^i(|\mathcal{O}|,R)$.
\end{theorem}

\chapter{Differential geometry of orbifolds}\label{section: diff geometry}

In this chapter we will see that much of the elementary objects and constructions from the differential topology and geometry of manifolds---beginning with the very notion of the differential of a smooth map, as a map between tangent (orbi)bundles, and including e.g. differential forms, integration, Stokes' theorem and De Rham cohomology---generalize to orbifolds. The presentation here follows mostly \cite{kleiner}, \cite{chenruan} and \cite{galazgarcia}.

\section{Tangent orbibundle and differentials}\label{section: tangent bundle}

Let $(\widetilde{U},H,\varphi)$ be an orbifold chart and consider the tangent bundle $T\widetilde{U}$. Since we have a smooth action, say $\mu$, of $H$ on $\widetilde{U}$, we can define a smooth $H$-action on $T\widetilde{U}$ by $h(\tilde{x},v)=(\mu(h,\tilde{x}),\dif(\mu^h)_{\tilde{x}}v)$. This gives us an orbifold chart $(T\widetilde{U},H,\pi)$, where $\pi$ is the quotient projection over $TU:=T\widetilde{U}/H$. Notice that the foot projection $T\widetilde{U}\to\widetilde{U}$ is equivariant, hence induces a projection $|p|:TU\to U\cong \widetilde{U}/H$. For $x=\varphi(\tilde{x})$ we have
$$|p|^{-1}(x)=\{H(z,v)\ |\ z=\tilde{x}\}\subset TU.$$
We claim that $|p|^{-1}(x)\cong T_{\tilde{x}}\widetilde{U}/\Gamma_{x}$ (recall: $\Gamma_x:=H_{\tilde{x}})$. In fact, we have $H(\tilde{x},v)=H(\tilde{x},w)$ if and only if there exists $h\in H$ such that $h(\tilde{x},v)=(\tilde{x},w)$, which happens if and only if $h\in H_{\tilde{x}}$ and $\dif(\mu^h)_{\tilde{x}}v=w$, which in turn is equivalent to $H_{\tilde{x}}v=H_{\tilde{x}}w$. So $H(\tilde{x},v)\mapsto\Gamma_xv$ is a well defined bijection $|p|^{-1}(x)\to T_{\tilde{x}}\widetilde{U}/\Gamma_{x}$ which is clearly continuous. Its inverse is the inclusion $T_{\tilde{x}}\widetilde{U}/\Gamma_{x}\to TU$, establishing the claim.

For an orbifold $\mathcal{O}$ with atlas $(\widetilde{U}_i,H_i,\varphi_i)$ we now glue the charts $(T\widetilde{U}_i,H_i,\pi_i)$ (with gluing maps given by the chart embeddings, as we did in Example \ref{example: doubling mirrors}) to obtain an orbifold $T\mathcal{O}$, the \textit{tangent bundle} of $\mathcal{O}$, which has
$$|T\mathcal{O}|=\frac{\displaystyle\bigsqcup (T\widetilde{U}_i/H_i)}{\sim}.$$
The projection $|p|$ locally lifts to $(\tilde{x},v)\to \tilde{x}$ (which is $\mathrm{id}_H$-equivariant), so it defines a smooth orbifold map. Let us summarize this.

\begin{proposition}[Tangent bundle]
The tangent bundle $T\mathcal{O}$ of an $n$-dimensional orbifold $\mathcal{O}$ is a $2n$-dimensional orbifold and the projection $p:T\mathcal{O}\to\mathcal{O}$ is a smooth orbifold map.
\end{proposition}

The \textit{tangent space} at $x=\varphi(\tilde{x})$, denoted $T_x\mathcal{O}$, is the vector space $T_{\tilde{x}}\widetilde{U}$ together with the induced linear action of $\Gamma_x$. The \textit{tangent cone} at $x$ is $C_x|\mathcal{O}|:=T_{\tilde{x}}\widetilde{U}/\Gamma_x$, which, as we saw, coincides with the fiber $|p|^{-1}(x)$ (see Figure \ref{tangentcone}). A \textit{tangent vector} is an equivalence class $[v]\in C_x|\mathcal{O}|$ (when there is no risk of confusion we will omit the brackets and write just $v\in C_x|\mathcal{O}|$).

\begin{figure}
\centering{
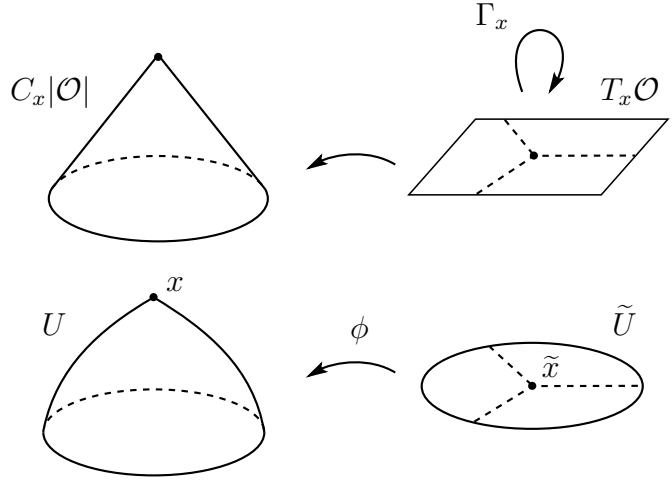}
\caption{Tangent space and tangent cone.}
\label{tangentcone}
\end{figure}

We can now define the differential of a smooth map $f:\mathcal{O}\to\mathcal{P}$. A local lift $\tilde{f}_x:(\widetilde{U},\Gamma_x,\phi)\to(\widetilde{V},\Gamma_{|f|(x)},\psi)$, defines an $\overline{f}_x$-equivariant map $\dif\tilde{f}_x:T_{\tilde{x}}\widetilde{U}\to T_{\tilde{f}_x}\widetilde{V}$ between the induced $T\mathcal{O}$ and $T\mathcal{P}$ charts. This gives a linear map $(\dif f_x,\overline{f}_x):T_x\mathcal{O}\to T_{f(x)}\mathcal{P}$, the \textit{differential of $\f$ at $x$}, and defines a smooth bundle map $\dif f:T\mathcal{O}\to T\mathcal{P}$, the \textit{differential} of $f$.

A \textit{vector field} on $\mathcal{O}$ is a section of $T\mathcal{O}$, that is, a smooth map $X:\mathcal{O}\to T\mathcal{O}$ such that $p\circ X=\mathrm{id}_{\mathcal{O}}$. We denote the $C^{\infty}(\mathcal{O})$-module of the smooth vector fields in $\mathcal{O}$ by $\mathfrak{X}(\mathcal{O})$. In terms of a chart $(\widetilde{U},H,\phi)$ and the induced chart for $TU$, a vector field $X$ defines an $H$-invariant vector field $\widetilde{X}$ in $\widetilde{U}$. In fact, there is a natural correspondence between vector fields on $\mathcal{O}$ and $\mathscr{H}_{\mathcal{O}}$-invariant vector fields on $U_\mathcal{O}$. Hence, by naturality $\mathfrak{X}(\mathcal{O})$ is closed under Lie brackets. For a vector field $X\in\mathfrak{X}(\mathcal{O})$ with an isolated zero at $x\in\mathcal{O}$, the \textit{index} of $X$ at $x$ is defined as
$$\ind_x(X):=\frac{1}{|\Gamma_x|}\ind_{\tilde{x}}(\widetilde{X}).$$
There is the following orbifold version of the Poincaré--Hopf index theorem due to I.~Satake.

\begin{theorem}[Poincaré--Hopf theorem for orbifolds {\cite[Theorem 3]{satake2}}]
Let $\mathcal{O}$ be a compact orbifold and $X\in\mathfrak{X}(\mathcal{O})$ a vector field with isolated zeros at $x_1,\dots,x_k$. Then
$$\chiorb(\mathcal{O})=\sum_{i=1}^k\ind_{x_i}(X).$$
\end{theorem}

Notice that not every tangent vector $[v]\in C_x|\mathcal{O}|$ admits an extension to a vector field on a neighborhood of $x$. In fact, only the vectors satisfying $[v]=\{v\}$, that is, the ones represented by vectors in $T_x\mathcal{O}^{\Gamma_x}$ can be extended.

\section{Suborbifolds, immersions and submersions}\label{section: suborbifolds, immersions and submersions}

We say that a smooth map $f:\mathcal{O}\to\mathcal{P}$ is an \textit{immersion} (\textit{submersion}) \textit{at $x\in\mathcal{O}$} if $\dif f_x$ is injective (surjective). Recall are considering $\dif f_x$ as a map $(T_x\mathcal{O},\Gamma_x)\to(T_{f(x)}\mathcal{P},\Gamma_{|f|(x)})$, so this means that there is a local lift $\tilde{f}_x:\widetilde{U}\to\widetilde{V}$ such that \emph{both} $\dif\tilde{f}_x$ and $\overline{f}_x$ are injective (surjective)\footnote{Notice that in case of immersions the injectivity of $\overline{f}_x$ is automatic.}. When $f$ is an immersion (submersion) at each point of $\mathcal{O}$ we say that $f$ is an \textit{immersion} (\textit{submersion}). Many results of the differential topology of manifolds generalize to the orbifold setting, as can be seen in \cite{borzellino2}, \cite{cooper}, \cite{chenruan}, \cite{galazgarcia} and \cite{kleiner}, for example. As an illustration we present the following.

\begin{proposition}[Local form of submersions {\cite[Lemma 2.5]{kleiner}}]\label{proposition: Local form of submersions}
Let $f:\mathcal{O}\to\mathcal{P}$ be a submersion at $x$. Then there is an orbifold $\f$, on which $\Gamma_{|f|(x)}$ acts\footnote{That is, there is an action on the underlying topological space which is smooth as a map between orbifolds. We will see this in more detail in Section \ref{section: actions on orbifolds}.}, and a chart $(\widetilde{V},\Gamma_{|f|(x)},\phi)$ around $|f|(x)$ such that $f$ is equivalent, near $x$, to the projection
$$\f\times\widetilde{V}/\!/\Gamma_{|f|(x)}\longrightarrow\widetilde{V}/\!/\Gamma_{|f|(x)}.$$
\end{proposition}

\begin{proof}
Let $\widetilde{f}_x:\widetilde{U}\to\widetilde{V}$ be a local lift of $f$ and consider the fiber $F:=(\widetilde{f}_x)^{-1}(\widetilde{f}_x(\tilde{x}))$, where $\tilde{x}\in\widetilde{U}$ is the lift of $x$. By hypothesis, $\widetilde{f}_x$ is a submersion, so we can suppose, after reducing $\widetilde{U}$ and $\widetilde{V}$ if necessary, that there is an $\overline{f}_x$-equivariant diffeomorphism $\widetilde{U}\to F\times \widetilde{V}$, such that the diagram
$$\xymatrix{
\widetilde{U} \ar[dr]_{\widetilde{f}_x} \ar[r] & F\times\widetilde{V} \ar[d]\\
 & \widetilde{V}}$$
commutes. Notice that $\ker(\overline{f}_x)$ acts on $F$. Setting $\f:=F/\!/\ker(\overline{f}_x)$ we then get the following commuting diagram
$$\xymatrix{
\widetilde{U}/\!/\ker(\overline{f}_x) \ar[dr] \ar[r] & \f\times\widetilde{V} \ar[d]\\
 & \widetilde{V}}$$
of orbifold smooth maps. Now, recalling that $\overline{f}_x$ is surjective, quotienting by $\Gamma_{|f|(x)}$ yields the commutative diagram
$$\xymatrix{
\widetilde{U}/\!/\Gamma_x \ar[dr] \ar[r] & (\f\times\widetilde{V})/\!/\Gamma_{|f|(x)} \ar[d]\\
 & \widetilde{V}/\!/\Gamma_{|f|(x)}}$$
whose horizontal line is a diffeomorphism.
\end{proof}

\begin{exercise}[Inverse function theorem for orbifolds]
Show that if $f:\mathcal{O}\to\mathcal{P}$ is smooth at $x$ and $\dif f_x$ is invertible\footnote{Again, this means both $\dif\tilde{f}_x$ and $\overline{f}_x$ are invertible, for some smooth lift.} then $f$ is an orbifold diffeomorphism near $x$.
\end{exercise}

A point $y\in|\mathcal{P}|$ is called a \textit{regular value} of $f:\mathcal{O}\to\mathcal{P}$ when $f$ is a submersion at every $x\in f^{-1}(y)$. Otherwise $y$ is a \textit{critical value} of $f$. Sard's Theorem directly generalizes to orbifolds:

\begin{theorem}[Sard's theorem for orbifolds {\cite[Theorem 4.1]{borzellino6}}]
The set of critical values of a smooth map $f:\mathcal{O}\to\mathcal{P}$ has measure zero and, hence, the set of regular values of $f$ is dense in $\mathcal{P}$.
\end{theorem}

A \textit{suborbifold} of an orbifold $\mathcal{O}$ is given by an orbifold $\mathcal{S}$ and an immersion $i:\mathcal{S}\to\mathcal{O}$ such that $|i|$ maps $|\mathcal{S}|$ homeomorphically onto its image in $|\mathcal{O}|$. Then from effectiveness $\overline{i}_x:\Gamma_x\to\Gamma_{|i|(x)}$ is injective for each $x\in|\mathcal{S}|$ (and not necessarily surjective). We will identify $\mathcal{S}$ with its image in $\mathcal{O}$. Let $(\widetilde{U},\Gamma_x,\phi)$ be an $\mathcal{O}$-chart around a point $x\in|\mathcal{S}|$ which is regular as a point of $\mathcal{S}$. The \textit{multiplicity of $\mathcal{S}$} as a suborbifold of $\mathcal{O}$ is $m_{\mathcal{S}}=\#\Gamma_x$.

We say that a suborbifold $\mathcal{S}$ is \textit{strong} or \textit{full} (following \cite{galazgarcia} and \cite{borzellino7}) when, for every $x\in\mathcal{S}$ and every $\mathcal{O}$-chart around $x$, the image of a lift $\widetilde{i}_x$ does not depend on the lift. This definition is equivalent to Thurston's definition of suborbifold in \cite{thurston}. That is, if $(\widetilde{U},\Gamma_x,\phi)$ is an $\mathcal{O}$-chart around $x\in\mathcal{S}$, then $\phi^{-1}(U\cap|\mathcal{S}|)$ is a $\Gamma_x$-invariant submanifold of $\widetilde{U}$.

\begin{example}[Open suborbifolds]\label{example: open suborbifolds}
If $U$ is an open subset of the underlying space of an orbifold $\mathcal{O}$ then the restriction of the charts of $\mathcal{O}$ provide a natural orbifold structure on $U$, that we denote $\mathcal{U}=\mathcal{O}|_U$. It is clear that, with the inclusion $i_{\mathcal{U}\mathcal{O}}:\mathcal{U}\to\mathcal{O}$, it is a (strong) suborbifold.
\end{example}

\begin{example}[Slices of products]
For a product orbifold $\mathcal{O}_1\times\mathcal{O}_k$, each slice
$$\{x_1\}\times\dots\times\mathcal{O}_i\times\dots\times\{x_n\}$$
is a strong suborbifold (with the natural inclusion).
\end{example}

Using Proposition \ref{proposition: Local form of submersions} one proves the following.

\begin{theorem}[Preimage theorem for orbifolds {\cite[Theorem 4.2]{borzellino6}}]
Let $\mathcal{O}$ and $\mathcal{P}$ be orbifolds with $\dim\mathcal{O}\geq\dim\mathcal{P}$, let $f:\mathcal{O}\to\mathcal{P}$ be a smooth map and $y\in|\mathcal{P}|$ a regular value of $f$. Then the preimage $\mathcal{S}=f^{-1}(y)$ is a strong suborbifold with $\dim\mathcal{S}=\dim\mathcal{O}-\dim\mathcal{P}$.
\end{theorem}

\section{Orbibundles and Frobenius' theorem}\label{sec: orbibundles}

As in the manifold case, there is a general notion of fiber bundle over an orbifold that allows one to carry over to this setting other usual differential geometric constructions. Let us begin with the general definition. 

Let $\mathcal{E}$ and $\mathcal{B}$ be smooth orbifolds. A smooth map $p:\mathcal{E}\to\mathcal{B}$ is a \textit{fiber orbibundle} if $|p|$ is surjective and there is a third orbifold $\f$ such that, for all $x\in|\mathcal{B}|$, there is an orbifold chart $(\widetilde{U},\Gamma_x,\varphi)$ around $x$, an action of $\Gamma_x$ on $\f\times\widetilde{U}$ and a diffeomorphism $(\f\times\widetilde{U})/\!/\Gamma_x\to\mathcal{E}|_{|p|^{-1}(U)}$ such that
$$\xymatrix{
(\f\times\widetilde{U})/\!/\Gamma_x \ar[r] \ar[d] & \mathcal{E} \ar[d]^{p}\\
\widetilde{U}/\!/\Gamma_x \ar[r] & \mathcal{B}}$$
is commutative\footnote{It is implied here that the $\Gamma_x$-action on $\f\times\widetilde{U}$ extends that on $\widetilde{U}$. As noted in \cite[p. 3]{biringer}, there was a typo in previous versions of this manuscript leading to a wrong interpretation of this definition, I thank the authors for pointing that out.}. A \textit{section} of $p$ is a smooth map $s:\mathcal{B}\to\mathcal{E}$ satisfying $p\circ s=\mathrm{id}_{\mathcal{B}}$.

When $\f$ is a $k$-dimensional vector space $V$ with a linear action of $\Gamma_x$ we say that $p$ is a \textit{vector orbibundle}.

\begin{example}[Coverings]
An orbifold covering $\widehat{\mathcal{O}}\to\mathcal{O}$ can be seen as an orbibundle with discrete fibers.
\end{example}

\begin{example}[Tangent bundle]
It is clear from what we saw in Section \ref{section: tangent bundle} that the tangent bundle $T\mathcal{O}$ of an orbifold $\mathcal{O}$ is a rank-$n$ vector orbibundle.
\end{example}

We end this section presenting the orbifold version of Frobenius' theorem. An \textit{involutive distribution} over an orbifold $\mathcal{O}$ is a subbundle $\mathcal{E}\subset T\mathcal{O}$ such that for any two sections $X$ and $Y$ of $\mathcal{E}$, the Lie bracket $[X,Y]$ is also a section of $\mathcal{E}$.

\begin{theorem}[Frobenius' theorem for orbifolds {\cite[Lemma 2.10]{kleiner}}]
Let $\mathcal{E}$ be an involutive distribution over an orbifold $\mathcal{O}$. Then for each $x\in|\mathcal{O}|$ there is a unique maximal suborbifold containing $x$ and tangent to $\mathcal{E}$.
\end{theorem}

\section{Integration and de Rham cohomology}

As we did to $T\mathcal{O}$, by gluing the local data we obtain \textit{$(j,i)$-tensor orbibundles} over $\mathcal{O}$, denoted $\bigotimes_j^i\mathcal{O}$, which have as suborbibundles the \textit{$i$-th exterior orbibundles} $\bigwedge^i\mathcal{O}$. Smooth sections of $\bigotimes_j^i\mathcal{O}$ and $\bigwedge^i\mathcal{O}$ yields us \textit{$(j,i)$-tensor} fields and \textit{$i$-forms} on $\mathcal{O}$, respectively. We denote the space of $i$-forms on $\mathcal{O}$ by $\Omega^i(\mathcal{O})$. Again, in a local chart $(\widetilde{U},H,\phi)$ they define $H$-invariant tensor fields in $\widetilde{U}$ and we have a correspondence between tensor fields on $\mathcal{O}$ and $\mathscr{H}_{\mathcal{O}}$-invariant tensor fields on $U_{\mathcal{O}}$. In particular, by naturality there is a well-defined \textit{exterior derivative}
$$d:\Omega^i(\mathcal{O})\longrightarrow\Omega^{i+1}(\mathcal{O}).$$
Similarly, if $X$ and $\xi$ are $\mathscr{H}_{\mathcal{O}}$-invariant vector and $(0,i)$-tensor fields on $U_{\mathcal{O}}$, respectively, then $\mathcal{L}_X\xi$ will also be $\mathscr{H}_{\mathcal{O}}$-invariant, hence Lie derivatives of tensor fields are well defined on $\mathcal{O}$.

On an oriented $n$-orbifold $\mathcal{O}$ we can define integration of $n$-forms. Let $(\widetilde{U},H,\phi)$ be an orbifold chart for $U$. Given a compactly supported $n$-form $\omega\in\Omega^n(\widetilde{U}/\!/H)$, i.e, an $H$-invariant compactly supported $n$-form $\tilde{\omega}\in\Omega^n(\widetilde{U})$, we define
$$\int_U\omega:=\frac{1}{|H|}\int_{\widetilde{U}}\tilde{\omega}.$$
For a general compactly supported $n$-form $\omega\in\Omega^n(\mathcal{O})$ we use partitions of unity to define the integral as a sum of these chart integrals.  

\begin{lemma}[Partitions of unity for orbifolds {\cite[Lemma 2.11]{kleiner}} {\cite[Lemma 4.2.1]{chenruan}}]
Any atlas of a smooth orbifold $\mathcal{O}$ admits a locally finite refinement $\{\widetilde{U}_i,H_i,\phi_i\}$ such that there exists a collection of functions $\xi_i\in C^\infty(\mathcal{O})$ satisfying
\begin{enumerate}[(i)]
\item $0\leq\xi_i\leq1$,
\item $\mathrm{supp}(\xi_i)\subset U_i$,
\item $\sum_i \xi_i(x)=1$, for all $x\in|\mathcal{O}|$.
\end{enumerate}
\end{lemma}
The proof of this lemma is similar to the case of manifolds: one finds the locally finite refinement by paracompactness, then work with $H_i$-invariant functions on $\widetilde{U}_i$.

Now, for a compactly supported $n$-form $\omega\in\Omega^n(\mathcal{O})$ we can define
$$\int_{\mathcal{O}}\omega:=\sum_j\int_{U_j}\xi_j\omega,$$
where $\{\xi_j\}$ is a partition of unity subordinated to an oriented atlas $\{(\widetilde{U}_j,H_j,\phi_j)\}$. This definition is independent of the choices involved (see \cite[p. 35]{adem}).

\begin{theorem}[Stokes' theorem for orbifolds]
Let $\mathcal{O}$ be an oriented $n$-dimensional orbifold with boundary and let $\omega\in\Omega^{n-1}(\mathcal{O})$ be compactly supported. Then
$$\int_{\mathcal{O}}d\omega=\int_{\partial\mathcal{O}}\omega.$$
\end{theorem}

\begin{exercise}
Prove Stokes' theorem for orbifolds by reducing to the classical Stokes' theorem using a partition of unity.
\end{exercise}

Given an orbifold $\mathcal{O}$, the cohomology groups of the complex
$$\cdots \stackrel{d}{\longrightarrow} \Omega^{i-1}(\mathcal{O}) \stackrel{d}{\longrightarrow} \Omega^i(\mathcal{O}) \stackrel{d}{\longrightarrow} \Omega^{i+1}(\mathcal{O}) \stackrel{d}{\longrightarrow}\cdots$$
are the \textit{de Rham cohomology groups} of $\mathcal{O}$, that we denote by $H_{\mathrm{dR}}^i(\mathcal{O})$. The rank $b_i=\rank H_{\mathrm{dR}}^i(\mathcal{O})$ is the \textit{$i$-th Betti number of $\mathcal{O}$}. As in the manifold case, these groups are invariant under homotopy equivalence. This follows from the result below, which can be seen as a version of the De Rham theorem for orbifolds, that asserts that they are isomorphic to the real singular cohomology groups of the quotient space $|\mathcal{O}|$ (see also \cite[Theorem 2.13]{adem}).

\begin{theorem}[De Rham theorem for orbifolds {\cite[Theorem 1]{satake}}]\label{theorem: Satake}
Let $\mathcal{O}$ be an orbifold. Then $H_{\mathrm{dR}}^i(\mathcal{O})\cong H^i(|\mathcal{O}|,\mathbb{R})$.
\end{theorem}

Virtually all the machinery involving differential forms on manifolds generalize to orbifolds. A particular case  is that, if $\mathcal{O}$ is an $n$-dimensional compact oriented orbifold, $\omega\otimes\eta\mapsto \int_{\mathcal{O}}\omega\wedge\eta$ induces a pairing
$$H^i_{\mathrm{dR}}(\mathcal{O})\otimes H^{n-i}_{\mathrm{dR}}(\mathcal{O})\longrightarrow \mathbb{R},$$
and one has the following.

\begin{theorem}[Poincaré duality for De Rham cohomology of orbifolds {\cite[Theorem 3]{satake}}]
Let $\mathcal{O}$ be an oriented, compact, $n$-dimensional orbifold. Then the above pairing is non-degenerate, hence $H^i_{\mathrm{dR}}(\mathcal{O})\cong H^{n-i}_{\mathrm{dR}}(\mathcal{O})$.
\end{theorem}

\section{Actions on orbifolds}\label{section: actions on orbifolds}

Let $G$ be a Lie group and $\mathcal{O}$ be an orbifold. We say that a smooth orbifold map $a:G\times\mathcal{O}\to\mathcal{O}$ is a \textit{smooth action} of $G$ on $\mathcal{O}$ if $|a|:G\times|\mathcal{O}|\to|\mathcal{O}|$ is a continuous action. All usual properties and notions of group actions are defined for $a$ in terms of $|a|$.

For each fixed $g\in G$, a smooth action defines a diffeomorphism $a^g:\mathcal{O}\to\mathcal{O}$. Hence the action respects the stratification of $\mathcal{O}$, so each orbit $Gx$ remains within a stratum $\Sigma_\alpha\ni x$. In fact, as in the manifold case, $a$ induces an injective immersion $G/G_x\to G_x$.

\begin{proposition}[Orbits are submanifolds {\cite[Lemma 2.11 and Corollary 2.14]{galazgarcia}}]\label{prop: fixed points suborbifolds}
Let $a:G\times\mathcal{O}\to\mathcal{O}$ be a smooth action. Then each orbit is a manifold and a strong suborbifold of $\mathcal{O}$. Moreover, if $G$ is compact and connected and acts effectively, then each connected component of the fixed point set $\mathcal{O}^G$ is also a strong suborbifold.
\end{proposition}

Suppose we have an effective action by a compact Lie group $G$ on $\mathcal{O}$ and let $x\in|\mathcal{O}|$. For a chart $(\widetilde{U},\Gamma_x,\phi)$ around $x\in|\mathcal{O}|$, consider
$$\widetilde{G}_x:=\{F_g:\widetilde{U}\to\widetilde{U}\}\ |\ \phi\circ F_g=a^g\circ\phi,\ g\in G_x\},$$
the collection of all possible lifts of $a^g$ near $x$. This is a Lie group, in fact an extension\footnote{That is, there is a short exact sequence $0\to\Gamma_x\to\widetilde{G}_x\to G_x\to 0$.} of $G_x$ by $\Gamma_x$, which acts on $\widetilde{U}$ satisfying $\widetilde{U}/\widetilde{G}_x=U/G_x$ (see \cite[Proposition 2.12]{galazgarcia}).

The tubular neighborhood theorem for smooth actions on manifolds generalizes as follows.

\begin{theorem}[Tubular neighborhood theorem for orbifolds {\cite[Proposition 2.3.7]{yeroshkin}} {\cite[Theorem 2.18]{galazgarcia}}]
Let $\mathcal{O}$ be an orbifold with a smooth effective action by a compact Lie group $G$. Then for each $x\in|\mathcal{O}|$ there is a $G_x$-invariant neighborhood of $Gx$ which is equivariantly diffeomorphic to
$$G\times_{G_x}(\nu_xGx/\!/\Gamma_x)\cong G\times_{\widetilde{G}_x}\nu_xGx.$$
\end{theorem}

One can now generalize Proposition \ref{prop: quotient orbifolds} to the following.

\begin{exercise}\label{exercise: quotient orbifold by action on orbifold}
Let $\mathcal{O}$ be an orbifold with an almost free, effective action by a compact Lie group $G$. Imitate the proof of Proposition \ref{prop: quotient orbifolds} to conclude that the quotient space $|\mathcal{O}|/G$ has a natural orbifold structure (which we denote by $\mathcal{O}/\!/G$, as in the case $\mathcal{O}=M$).
\end{exercise}

An action $a:G\times\mathcal{O}\to\mathcal{O}$ induces a Lie algebra anti-homomorphism $\mathfrak{g}\ni X\mapsto X^\#\in\mathfrak{X}(\mathcal{O})$, where $X^\#$ is the fundamental vector field corresponding to $X$: 
$$X^\#(x) =\left.\od{}{t} a(\exp(tX),x) \right|_{t=0}.$$
We say that a differential form $\omega\in\Omega(\mathcal{O})$ is \textit{horizontal} when $\iota(X^\#)\omega=0$ for all $X\in\mathfrak{g}$. 
Moreover, $\omega$ is \textit{basic} when it is horizontal and invariant (i.e., $(a^g)^*\omega=\omega$ for all $g\in G$). When the action is as in Exercise \ref{exercise: quotient orbifold by action on orbifold} the projection $\pi:\mathcal{O}\to \mathcal{O}/\!/G$ induces an isomorphism $\pi^*:\Omega(\mathcal{O}/\!/G)\to \Omega_{\bas}(\mathcal{O})$, where $\Omega_{\bas}(\mathcal{O})$ is the sub-complex of basic forms, so we have:

\begin{proposition}\label{proposition: basic cohomology and quotients}
Let $\mathcal{O}$ be an orbifold with an almost free, effective action by a compact Lie group $G$. Then $\pi^*:H_{\mathrm{dR}}(\mathcal{O}/\!/G)\to H_{\bas}(\mathcal{O})$ is an isomorphism.
\end{proposition}

\section{Equivariant cohomology of orbifolds}

If a Lie group $G$ acts on $\mathcal{O}$ (not necessarily almost freely) we can consider the equivariant cohomology of $\mathcal{O}$. This is a cohomology theory which will capture information not only of the topology of $\mathcal{O}$ but also of the action, its orbit space and stabilizers\footnote{We have already seen a special case of this object in Section \ref{section: orbifold cohomology and homotopy}}. More precisely, consider the universal principal $G$-bundle $EG\to BG$ \cite[Section 4.11]{husemoller} and form the Borel construction $\mathcal{O}_G:=EG\times_G|\mathcal{O}|$. The \textit{$G$-equivariant cohomology of $\mathcal{O}$} is defined as
$$H_G(\mathcal{O}):=H(\mathcal{O}_G,\mathbb{R}),$$
that is, the singular cohomology of $\mathcal{O}_G$ with coefficients in $\mathbb{R}$.

Alternatively, we have the induced infinitesimal action of the Lie algebra $\mathfrak{g}$ of $G$ on $\mathcal{O}$ given by $\mathfrak{g}\ni X\longmapsto X^{\#}\in\mathfrak{X}(\mathcal{O})$. This defines, for each $X\in\mathfrak{g}$, operators $\iota_X$, $\mathcal{L}_X$ on $\Omega^*(\mathcal{O})$ which, together with $d$, endow $\Omega^*(\mathcal{O})$ with the structure of a $\mathfrak{g}^\star$-algebra (see, for example, \cite[Chapter 2]{guillemin}) and so enables us to also study the $\mathfrak{g}$-equivariant cohomology of $\mathcal{O}$. More precisely, we consider the \textit{Cartan complex}
$$C_{\mathfrak{g}}(\mathcal{O}):=(\sym(\mathfrak{g}^\vee)\otimes \Omega^*(\mathcal{O}))^{\mathfrak{g}},$$
where $\sym(\mathfrak{g}^\vee)$ is the symmetric algebra over the dual $\mathfrak{g}^\vee$ of $\mathfrak{g}$. That is, $C_{\mathfrak{g}}(\mathcal{O})$ consists of those $\omega\in\sym(\mathfrak{g}^\vee)\otimes \Omega^*(\mathcal{O})$ that satisfy $\mathcal{L}_X\omega=0$ for all $X\in\mathfrak{g}$. An element $\omega\in C_{\mathfrak{g}}(\mathcal{O})$ can be viewed as a $\mathfrak{g}$-equivariant\footnote{With respect to the coadjoint acion of $G$ on $\sym(\mathfrak{g}^\vee)$. In particular, if $G$ is Abelian then an element $C_{\mathfrak{g}}(\mathcal{O})$ is just a polynomial map $\omega:\mathfrak{g}\to \Omega^*(\mathcal{O})^G$} polynomial map $\omega:\mathfrak{g}\to \Omega^*(\mathcal{O})$. With this in mind, the \textit{equivariant differential} $d_\mathfrak{g}$ is defined as
$$(d_\mathfrak{g}\omega)(X)=d(\omega(X))-\iota_X(\omega(X)).$$
It is a degree $1$ derivation with respect to the grading $C_{\mathfrak{g}}^n(\mathcal{O})=\bigoplus_{2k+l=n}(\sym_k(\mathfrak{g}^\vee)\otimes \Omega^l(\mathcal{O}))^{\mathfrak{g}}$ and satisfy $d_\mathfrak{g}^2=0$. We define the \textit{$\mathfrak{g}$-equivariant cohomology} of $\mathcal{O}$ as the cohomology of the complex $(C_{\mathfrak{g}}(\mathcal{O}),d_\mathfrak{g})$, denoted
$$H_\mathfrak{g}(\mathcal{O}):=H(C_{\mathfrak{g}}(\mathcal{O}),d_\mathfrak{g}).$$
There is a natural structure of $\sym(\mathfrak{g}^\vee)^{\mathfrak{g}}$-algebra on $H_\mathfrak{g}(\mathcal{O})$, induced by the inclusion $\sym(\mathfrak{g}^\vee)^{\mathfrak{g}}\to C_{\mathfrak{g}}(\mathcal{O})$ given by $f\mapsto f\otimes 1$.

The orbifold version of the equivariant De Rham theorem states that $H_G(\mathcal{O})$ and $H_\mathfrak{g}(\mathcal{O})$ are the same, provided $G$ is compact and connected.

\begin{theorem}[Equivariant De Rham theorem for orbifolds {\cite[Theorem 3.5]{caramello2}}]\label{thrm: equivariant De Rham theorem for orbifolds}
Let a connected, compact Lie group $G$ with Lie algebra $\mathfrak{g}$ act on an $n$-dimensional orbifold $\mathcal{O}$. Then
$$H_G(\mathcal{O})\cong H_\mathfrak{g}(\mathcal{O})$$
as $\sym(\mathfrak{g}^\vee)^{\mathfrak{g}}$-algebras.
\end{theorem}

A \textit{connection} for the $\mathfrak{g}$-action on $\mathcal{O}$ is an invariant element $\theta\in \Omega^1(\mathcal{O})\otimes\mathfrak{g}$ such that $\iota_X\theta=X$ for all $X\in\mathfrak{g}$ or, equivalently, a collection $\theta^i\in \Omega^1(\mathcal{O})$ satisfying $\iota_{X_j}\theta^i=\delta^i_j$, for some basis $\{X_i\}$ of $\mathfrak{g}$, such that
$$C=\spannn\{\theta^1,\dots,\theta^{\dim\mathfrak{g}}\}$$
is $\mathfrak{g}$-invariant. If $G$ is connected, compact and acts almost freely on $\mathcal{O}$, then there is a connection for its induced $\mathfrak{g}$-action (see \cite[Proposition 3.2]{caramello2}). The \textit{curvature} $\mu^\theta:=\dif \theta+\frac{1}{2}[\theta,\theta]_{\mathfrak{g}}\in\Omega^2(\mathcal{O})\otimes\mathfrak{g}$ of $\theta$ defines the Cartan map $C^\theta$:
$$C_{\mathfrak{g}}(\mathcal{O})\ni\omega\mapsto \hor_\theta(\omega(\mu^\theta))\in \Omega_{\bas}(\mathcal{O}),$$
where $\omega(\mu^\theta)$ denotes the differential form obtained from $\omega$ by substituting $\mu^\theta$ for the $\mathfrak{g}$-variable, and $\hor_\theta$ is the horizontal projection with respect to $\theta$. The restriction of $C^\theta$ to $\sym(\mathfrak{g}^\vee)^{\mathfrak{g}}$ is the Chern--Weil homomorphism of the $G$-action on $\mathcal{O}$ (cf. the construction of the Euler class at the end of Section \ref{section:Orbifold Versions of Classical Theorems}). The following theorem is a classical result on equivariant cohomology of $\mathfrak{g}^\star$-algebras (see, e.g., \cite[Theorem 5.2]{meinrenken2}) adapted to our setting.

\begin{theorem}
If $G$ is a connected, compact Lie group acting almost freely on an orbifold $\mathcal{O}$, then the Cartan map (with respect to any $\mathfrak{g}$-connection) descends to an isomorphism
$$H_{\mathfrak{g}}(\mathcal{O})\cong H_{\bas}(\mathcal{O}).$$
\end{theorem}
Hence, in the above situation we have, in view of Proposition \ref{proposition: basic cohomology and quotients} and Theorem \ref{thrm: equivariant De Rham theorem for orbifolds}, $H_G(\mathcal{O})\cong H_{\mathrm{dR}}(\mathcal{O}/\!/G)$.

A remarkable feature of equivariant cohomology is the Borel--Hsiang localization theorem, which, roughly speaking, asserts that the torsion-free part of the structure of $H_{T}(X)$, for a torus space $X$, can be recovered from the fixed point set $X^T$. To introduce this theorem it will be useful to recall the notion of localization from commutative algebra (see, for example, \cite[Chapter 2]{eisenbud}). Given an $R$-module $A$ and a multiplicative subset $S\subset R$ we denote the \textit{localization of $A$ at $S$} by $S^{-1}A$, which consists of $(A\times S)/\sim$, where $(a,s)\sim (a',s')$ if there is $r\in S$ such that $r(s'a+sa')=0$. One think of an equivalence class $(a,s)$ as a fraction $a/s$. In fact, $S^{-1}A$ is an $S^{-1}R$-module with the usual operation rules for fractions. A map of $R$-modules $\varphi:A\to B$ induces a map of $S^{-1}R$-modules $S^{-1}\varphi:S^{-1}A\to S^{-1}B$ by $a/s\mapsto \varphi(a)/s$.

\begin{theorem}[Borel--Hsiang localization for orbifolds {\cite[Corollary 3.1.8]{allday}}]\label{thrm: Borel-Hsiang local for orbifolds}
Let $\mathcal{O}$ be a compact orbifold acted upon by a torus $T$. Then the inclusion $\mathcal{O}^{T}\hookrightarrow\mathcal{O}$ induces an isomorphism of $S^{-1}\sym(\mathfrak{t}^\vee)$-algebras
$$S^{-1}H_{T}(\mathcal{O}) \longrightarrow S^{-1}H_{T}(\mathcal{O}^{T}),$$
where $S=\sym(\mathfrak{t}^\vee)\setminus \{0\}$.
\end{theorem}

We mention here, without going into much details, that this theorem has a ``concrete counterpart'' expressed in terms of integration of $d_{\mathfrak{t}}$-closed equivariant forms, known as the Atiyah--Bott--Berligne--Vergne localization formula.

\begin{theorem}[ABBV localization for orbifolds {\cite[Theorem 2.1]{meinrenken}}]
Let $\mathcal{O}$ be a connected, compact, oriented orbifold with a smooth action of a torus $T$ and let $[\omega]\in H_{\mathfrak{t}}(\mathcal{O})$. Then
$$\int_\mathcal{O}\omega=\sum_{\mathcal{C}}\frac{1}{m_{\mathcal{C}}}\int_{\mathcal{C}}\frac{i_{\mathcal{C}}^*\omega}{\mathrm{e}_{\mathfrak{t}}(\nu_{\mathcal{C}})}$$
in the fraction field $(\sym(\mathfrak{t}^\vee)\setminus \{0\})^{-1}\sym(\mathfrak{t}^\vee)$, where the sum runs over the connected components $\mathcal{C}$ of $\mathcal{O}^T$ and $\mathrm{e}_{\mathfrak{t}}(\nu_{\mathcal{C}})$ is the equivariant Euler form of the normal bundle $\nu_{\mathcal{C}}$.
\end{theorem}

\chapter{Riemannian geometry of orbifolds}\label{section: riemannian orbifolds}

In this final chapter we will endow orbifolds with riemannian metrics and see generalizations of the basic elements of Riemannian geometry to this setting. In the last section we present a collection of orbifold versions of classical theorems of Riemannian geometry.

\section{Riemannian metrics}\label{section:riemannian metrics on orbifolds}

A \textit{Riemannian metric} on an orbifold $\mathcal{O}$ is a symmetric, positive tensor field $\mathrm{g}\in\bigotimes_0^2(\mathcal{O})$. The pair $(\mathcal{O},\mathrm{g})$ is then a \textit{Riemannian orbifold}. As already mentioned for tensor fields in general, in a chart $(\widetilde{U},H,\phi)$ a Riemannian metric yields an $H$-invariant Riemannian metric on $\widetilde{U}$, and so we have an induced inner product $\proin{\cdot}{\cdot}_x:=\mathrm{g}_x(\cdot,\cdot)$ on $T_x\mathcal{O}$ for each $x$. Also, the embeddings between charts of $\mathcal{O}$ become isometries and the pseudogroup $\mathscr{H}_{\mathcal{O}}$ becomes, with the induced Riemannian metric in $U_\mathcal{O}$, a pseudogroup of local isometries. Hence, the Levi-Civita connection $\nabla$ on $U_\mathcal{O}$ is invariant by the changes of charts and thus can be seen as a covariant derivative on $T\mathcal{O}$. Parallel transport along $\mathscr{H}_{\mathcal{O}}$-paths is defined accordingly, and when restricted to $\mathscr{H}_{\mathcal{O}}$-loops based at $\tilde{x}$ they define the holonomy group $\mathrm{Hol}_x<\mathrm{O}(T_x\mathcal{O})$ (up to conjugation with elements of $\Gamma_x$). The metric defines a natural Riemannian volume measure $\vol$ on $\mathcal{O}$ which, for a measurable subset $K\subset U$ on a chart $(\widetilde{U},H,\phi)$, is given by
$$\vol(K)=\frac{1}{|H|}\vol_{\widetilde{U}}\phi^{-1}(K),$$
where $\vol_{\widetilde{U}}$ is the usual Riemannian volume measure induced on $\widetilde{U}$ by $\mathrm{g}$. In particular, $\vol(|\mathcal{O}|)$ coincides with the volume of the Riemannian manifold $\mathcal{O}_\mathrm{reg}$. 

\begin{proposition}[Existence of Riemannian metrics {\cite[Proposition 2.20]{mrcun}}]\label{prop: orbifold admits a Riemannian metric}
Any smooth orbifold admits a Riemannian metric.
\end{proposition}

\begin{proof}
Choose a locally finite atlas $\mathcal{A}=\{\widetilde{U}_i,H_i,\phi_i\}$, a subordinate partition of unity $\xi_i\in C^\infty(\mathcal{O})$ and an arbitrary Riemannian metric $\proin{\cdot}{\cdot}^i$ on each $\widetilde{U}_i$. Now define a new Riemannian metric $\mathrm{g}^i$ as follows: for each $\tilde{x}\in\widetilde{U}_i$ and each $v,w\in T_{\tilde{x}}\widetilde{U}_i$, put
$$\mathrm{g}^i_{\tilde{x}}(v,w):=\sum_j\xi_j(\phi_i(\tilde{x}))\sum_{h\in H_j}\proin{\dif(h\circ\lambda_j)_{\tilde{x}}v}{\dif(h\circ\lambda_j)_{\tilde{x}}w}^j_{h\lambda_j(\tilde{x})},$$
where $\lambda_j:(\widetilde{V},(H_i)_{\widetilde{V}},(\phi_i)|_{\widetilde{V}})\hookrightarrow (\widetilde{U}_j,H_j,\phi_j)$ is any chart embedding defined on an open $H_i$-invariant neighborhood of $\widetilde{V}\subset\widetilde{U}_i$ of $\tilde{x}$. The reader is invited to check that this defines the desired metric.\end{proof}

\begin{exercise}
Check that the collection $\mathrm{g}^i$ defines an $\mathscr{H}_{\mathcal{A}}$-invariant Riemannian metric on $U_\mathcal{A}$, and hence a Riemannian metric on $\mathcal{O}$.
\end{exercise}

A map $f:(\mathcal{O},\mathrm{g}^{\mathcal{O}})\to(\mathcal{P},\mathrm{g}^{\mathcal{P}})$ is a \textit{local isometry} if $f^*(\mathrm{g}^{\mathcal{P}})=\mathrm{g}^{\mathcal{O}}$ and $\dim\mathcal{O}=\dim\mathcal{P}$. If $\f$ is also a diffeomorphism then it is an \textit{isometry}. One has, in analogy to the manifold case, the following.

\begin{theorem}[Local isometries are coverings {\cite[Lemma 2.18]{kleiner}}]
If $f:\mathcal{O}\to\mathcal{P}$ is a local isometry, with $\mathcal{O}$ complete\footnote{See Section \ref{section: length structure}.} and $\mathcal{P}$ connected, then $f$ is a covering map.
\end{theorem}

The presence of a Riemannian metric on $\mathcal{O}$ enables us to define the orthonormal frame bundle of $\mathcal{O}$, as follows. If $(\widetilde{U},H,\phi)$ is a chart of $\mathcal{O}$, consider the orthogonal frame bundle $\widetilde{U}^\Yup$ with the induced action of $H$ by $h(\tilde{x},B)=(h(\tilde{x}),\dif h_{\tilde{x}}B)$. This is actually a free action, therefore $\widetilde{U}^\Yup/H$ is a manifold that inherits a proper, effective and almost free $\mathrm{O}(n)$-action from the action of $\mathrm{O}(n)$ on $\widetilde{U}^\Yup$. Taking the quotient by this action we obtain the natural projection $\widetilde{U}^\Yup/H\to U$. The manifolds $\widetilde{U}^\Yup/G$ glue together to form a manifold $\mathcal{O}^\Yup$, the \textit{orthonormal frame bundle} of $\mathcal{O}$. With the natural projection, it defines an orbibundle $\mathcal{O}^\Yup\to\mathcal{O}$ (for more details, see \cite{adem}, Section 1.3, and \cite{mrcun}, Section 2.4).

An orientation of $\mathcal{O}$ corresponds to a decomposition $\mathcal{O}^\Yup=\mathcal{O}^\Yup_+\sqcup\mathcal{O}^\Yup_-$. Then $\mathcal{O}^\Yup_+$ is an $\mathrm{SO}(n)$-orbibundle over $\mathcal{O}$. A key point of the construction above is that the quotient orbifold $\mathcal{O}^\Yup_+/\!/ \mathrm{SO}(n)$ is isomorphic to $\mathcal{O}$ (see \cite{mrcun}, Proposition 2.22). The orientability is needed so that we have an action of the \emph{connected} Lie group $\mathrm{SO}(n)$ inheriting $\mathcal{O}$ as a quotient, which ensures that the holonomy of an orbit matches the corresponding isotropy group. A similar construction can be carried over for non-orientable orbifolds by first taking the complexification $\widetilde{U}^\Yup\otimes\mathbb{C}$, which leads to an $\mathrm{U}(n)$-orbibundle $\mathcal{O}^\Yup_{\mathbb{C}}$ \label{page: orbi frame bundle} over $\mathcal{O}$. Moreover, the Riemannian metric $\mathrm{g}$ on $\mathcal{O}$ induces a Riemannian metric on $\mathcal{O}^\Yup_{\mathbb{C}}$ such that $\mathrm{U}(n)$ acts by isometries and $\mathcal{O}$ is isometric to the quotient $\mathcal{O}^\Yup_{\mathbb{C}}/\!/\mathrm{U}(n)$. Hence there is the following converse to Proposition \ref{prop: quotient orbifolds}.

\begin{proposition}[Every orbifold is a quotient {\cite[Proposition 5.21]{alex}} {\cite[Proposition 2.23]{mrcun}}]\label{proposition: every orbifold is a quotient}
Every Riemannian orbifold is isometric to the quotient space of an almost free isometric action of a compact connected Lie group.
\end{proposition}

\section{Geodesics and the induced length structure}\label{section: length structure}

If $\mathcal{O}$ is Riemannian and $\gamma:[a,b]\to\mathcal{O}$ is a piecewise smooth curve, its \textit{length} can be defined as
$$\ell(\gamma):=\int_a^b\|\gamma'(t)\|dt.$$
We induce the length structure\footnote{Recall that a metric space is a \textit{length space} when the distance between any two points can be realized as the infimum of the lengths of all rectifiable curves connecting those points.} $d(x,y)=\inf\ell(\gamma)$ on $|\mathcal{O}|$, where the infimum is taken amongst all piecewise smooth curves connecting $x$ and $y$. We can then consider the \textit{diameter} of $\mathcal{O}$, i.e., the diameter of $(|\mathcal{O}|,d)$, denoted $\diam(|\mathcal{O}|)$. We say that $\mathcal{O}$ is \textit{complete} when $(|\mathcal{O}|,d)$ is a complete metric space.

\begin{remark}[Orbifolds as metric spaces]\label{Remark: orbifolds as metric spaces}
There is an alternative definition of Riemannian orbifolds in terms of metric spaces, due to A.~Lytchak. In fact, one could define an $n$-dimensional Riemannian orbifold as a length space $\mathcal{O}$ such that for any point $x\in\mathcal{O}$ there exists an open neighborhood $U\ni x$, a connected $n$-dimensional Riemannian manifold $M$ and a finite group $G<\mathrm{Iso}(M)$ such that $U$ and $M/G$ are isometric as metric spaces. Details on this approach can be seen in \cite{lange}.
\end{remark}

A smooth curve $\gamma:I\to\mathcal{O}$ is a \textit{geodesic} if $\nabla_{\gamma'}\gamma'=0$, that is,  if it lifts in local charts to curves satisfying the geodesic equation. Any locally minimizing curve between two points is a geodesic.

\begin{theorem}[Hopf--Rinow theorem for orbifolds]
For a connected Riemannian orbifold $\mathcal{O}$, the following are equivalent:
\begin{enumerate}[(i)]
\item $\mathcal{O}$ is complete,
\item every closed bounded set in $|\mathcal{O}|$ is compact,
\item\label{geodesic completeness} for any $x\in|\mathcal{O}|$ and any $v\in C_x|\mathcal{O}|$ there is a unique geodesic $\gamma:\mathbb{R}\to\mathcal{O}$ with $|\gamma|(0)=x$ and $\gamma'(0)=v$,
\end{enumerate}
Moreover, if one (and hence all) of those properties hold, then any two points on $|\mathcal{O}|$ can be connected by a minimizing geodesic.
\end{theorem}

The proof that completeness implies geodesic completeness (the property in item \eqref{geodesic completeness}) is similar to the classical proof of this fact for manifolds, as indicated in \cite[Lemma 2.16]{kleiner}. The other equivalences follow from Cohn-Vossen's generalization of the Hopf--Rinow theorem to length spaces \cite{cohn}.

We can therefore define the exponential map as follows. Given $x\in|\mathcal{O}|$ and $v\in C_x|\mathcal{O}|$, let $\gamma_c:[0,1]\to\mathcal{O}$ be the unique geodesic satisfying $|\gamma|_c(0)=x$ and $\gamma'_c(0)=v$. Define $|\exp|(x,v)=(x,|\gamma_c|(1))\in |\mathcal{O}|\times|\mathcal{O}|$, which is continuous and admits smooth lifts (see \cite[Proposition 4.2.8]{chenruan}), hence a smooth orbifold map $\exp:T\mathcal{O}\to\mathcal{O}\times\mathcal{O}$. Its restriction to each $T_x\mathcal{O}$ gives a smooth orbifold map $\exp_x:T_x\mathcal{O}\to\mathcal{O}$ so that $(x,|\exp_x|(v))=|\exp|(x,v)$.

\section{Orbifold versions of classical theorems}\label{section:Orbifold Versions of Classical Theorems}

Many results in the Riemannian geometry of manifolds generalize to orbifolds, as can be seen, for example, in \cite{borzellino3}, \cite{borzellino4}, \cite[Section 4.2]{chenruan}, \cite[Section 2.5]{kleiner} and \cite[Section 2.3]{yeroshkin}. We begin by citing the following version of De Rham's decomposition theorem.

\begin{theorem}[De Rham decomposition theorem for orbifolds {\cite[Lemma 2.19]{kleiner}}]
Let $\mathcal{O}$ be a connected, simply-connected and complete Riemannian orbifold. Suppose that for some $x\in \mathcal{O}_\mathrm{reg}$ there is an orthogonal splitting $T_x\mathcal{O}=V_1\oplus V_2$ which is invariant by holonomy around loops based at $x$. Then $\mathcal{O}$ splits isometrically as $\mathcal{V}_1\times \mathcal{V}_2$, so that $T_{x_1}\mathcal{V}_1= V_1$ and $T_{x_2}\mathcal{V}_2= V_2$, for $x=(x_1,x_2)$.
\end{theorem}

A vector field $X$ on $\mathcal{O}$ is a \textit{Killing vector field} if $\mathcal{L}_X\mathrm{g}=0$, which means that the local flows of $X$ act by isometries. There is the following analogue of the Myers--Steenrod Theorem.

\begin{theorem}[Myers--Steenrod theorem for orbifolds {\cite[Theorem 1]{bagaev}}]\label{theorem: Myers-Steenrod for orbifolds}
Let $\mathcal{O}$ be a connected $n$-dimensional Riemannian orbifold. With the compact-open topology, the isometry group $\mathrm{Iso}(\mathcal{O})$ of $\mathcal{O}$ has a Lie group structure with which it acts smoothly and properly on $\mathcal{O}$ (on the left). Moreover, if $\mathcal{O}$ is complete then the Lie algebra $\mathfrak{iso}(\mathcal{O})$ of $\mathrm{Iso}(\mathcal{O})$ is isomorphic to the opposite algebra of the Lie algebra of Killing vector fields.
\end{theorem}

We define the curvature tensor $R$ of $\mathcal{O}$ as the curvature tensor of the Levi-Civita connection $\nabla$ on $U_{\mathcal{O}}$. Derived curvature notions, such as sectional and Ricci curvatures (denoted $\sec_{\mathcal{O}}$ and $\ric_{\mathcal{O}}$, respect.), are defined accordingly. Let us now see some orbifold versions of classical Riemannian Geometry involving curvature hypotheses, beginning with the Bonnet--Myers Theorem for orbifolds.

\begin{theorem}[Bonnet--Myers theorem for orbifolds {\cite[Corollary 21]{borzellino3}}]\label{theorem: bonnet-myers for orbifolds}
Let $\mathcal{O}$ be a connected, complete $n$-dimensional Riemannian orbifold satisfying $\ric_{\mathcal{O}}\geq (n-1)/r^2$ for some $r>0$. Then $\diam(|\mathcal{O}|)\leq \pi r$. In particular, $|\mathcal{O}|$ is compact.
\end{theorem}

Since this result also applies to the universal covering $\widehat{\mathcal{O}}$ endowed with the pullback metric, it follows that $\piorb(\mathcal{O})$ (and hence $\pi_1(|\mathcal{O}|)$) must be finite. This also implies $b_1(\mathcal{O})=0$, since $H_i(|\mathcal{O}|,\mathbb{Z})=\abel(\pi_1(|\mathcal{O}|))$ is torsion. There is also an orbifold version of the Synge--Weinstein Theorem due to D.~Yeroshkin.

\begin{theorem}[Synge--Weinstein theorem for orbifolds {\cite[Theorem 2.3.5]{yeroshkin}}]\label{theorem: synge-weinstein for orbifolds}
Let $\mathcal{O}$ be a connected, compact, oriented, $n$-dimensional orbifold with $\sec_{\mathcal{O}}>0$ and let $f\in\mathrm{Iso}(\mathcal{O})$. Suppose that $f$ preserves orientation if $n$ is even and reverses orientation if $n$ is odd. Then $f$ has a fixed point.
\end{theorem}

Moreover, the classical Synge's theorem has a generalization to orbifolds.

\begin{theorem}[Synge's theorem for orbifolds {\cite[Corollary 2.3.6]{yeroshkin}}]\label{theorem: synge for orbifolds}
Let $\mathcal{O}$ be a compact orbifold with $\sec_{\mathcal{O}}>0$. Then
\begin{enumerate}[(i)]
\item if $\dim\mathcal{O}$ is even and $\mathcal{O}$ is orientable, then $|\mathcal{O}|$ is simply-connected, and
\item if $\dim\mathcal{O}$ is odd and $\mathcal{O}$ is locally orientable, then $\mathcal{O}$ is orientable.
\end{enumerate}
\end{theorem}

The Cheeger--Gromoll splitting theorem was generalized to orbifolds by J.~Borzellino and S.~Zhu. Recall that a \textit{line} on a Riemannian orbifold is a unit speed geodesic $\gamma:\mathbb{R}\to\mathcal{O}$ such that $d(\gamma(s),\gamma(t))=|s-t|$ for any $s,t\in\mathbb{R}$.

\begin{theorem}[Splitting theorem for orbifolds {\cite[Theorem 1]{borzellino8}}]
Let $\mathcal{O}$ be a complete Riemannian orbifold with $\ric_{\mathcal{O}}\geq 0$. If $\mathcal{O}$ has a line, then it is isometric to $\mathcal{P}\times\mathbb{R}$, where $\mathcal{P}$ is a complete Riemannian orbifold with $\ric_{\mathcal{P}}\geq 0$.
\end{theorem}

There is also a generalization on Cheng's sphere theorem, due to J.~Borzellino:

\begin{theorem}[Cheng's theorem for orbifolds {\cite[Theorems 1 and 2]{borzellino3}}]
Let $\mathcal{O}$ be a connected, complete, $n$-dimensional Riemannian orbifold satisfying $\ric_{\mathcal{O}}\geq n-1$ and $\mathrm{diam}(|\mathcal{O}|)=\pi$. Then $\mathcal{O}$ is a quotient of $\mathbb{S}^n$ by a finite group $\Gamma\in\mathrm{O}(n+1)$. In particular, $\mathcal{O}$ is good.
\end{theorem}

Püttmann and Searle proved in \cite[Theorem 2]{puttmann} that the famous Hopf conjecture --- that the Euler characteristic of a positively curved manifold is positive --- holds for manifolds with a large symmetry rank. By the \textit{symmetry rank}, $\symrank(\mathcal{O})$, of a Riemannian orbifold $\mathcal{O}$ we mean the maximal dimension of an Abelian subalgebra of $\mathfrak{iso}(\mathcal{O})$.

\begin{theorem}[Püttmann--Searle theorem for orbifolds {\cite[Proposition 8.8]{caramello}}]\label{theorem: putmann searle for orbifolds}
Let $\mathcal{O}$ be a connected, orientable, compact Riemannian orbifold. If $n=\dim(\mathcal{O})$ is even, $\sec_{\mathcal{O}}>0$ and $\symrank(\mathcal{O})\geq n/4-1$, then $\chi(|\mathcal{O}|)>0$.
\end{theorem}

Recall that the growth $\#_S$ of a finitely generated group $G=\langle S \rangle$, for $S=\{g_1,\dots,g_k\}$, is the function that associates to $j\in\mathbb{N}$ the number of distinct elements of $G$ which one can express as words of length at most $j$, using the alphabet $\{g_1,\dots,g_k,g_1^{-1},\dots,g_k^{-1}\}$. We say that $G$ has \textit{polynomial growth} when $\#_S(j)\leq C(j^k+1)$, for some $C,k\leq\infty$, and that $G$ has \textit{exponential growth} when $\#_S(j)\geq \alpha^j$ for some $\alpha>1$. Those properties are independent of the chosen set of generators (see \cite[Lemma 1]{milnor}). In \cite{milnor}, Milnor shows that the curvature of a compact Riemannian manifold influences the growth rate of its fundamental group. For orbifolds we have the following analog.

\begin{theorem}[Milnor's theorem for orbifolds {\cite[Proposition 10]{borzellino5}} {\cite[Theorem 2.3]{caramello2}}]\label{theorem: Milnor growth for orbifolds}
Let $\mathcal{O}$ be a connected, compact Riemannian orbifold. Then
\begin{enumerate}[(i)]
\item \label{item milnor borzellino} if $\ric_{\mathcal{O}}\geq 0$, then $\piorb(\mathcal{O})$ has polynomial growth, and
\item \label{item milnor negative} if $\sec_{\mathcal{O}}<0$, then $\piorb(\mathcal{O})$ has exponential growth.
\end{enumerate}
\end{theorem}

Concerning non-positively curved orbifolds, there is the following orbifold version of the Cartan--Hadamard theorem, which in \cite{bridson} is attributed to M.~Gromov.

\begin{theorem}[Cartan--Hadamard theorem for orbifolds {\cite[Corollary 2.16]{bridson}}]
Every connected, complete Riemannian orbifold $\mathcal{O}$ with $\sec_{\mathcal{O}}\leq0$ is good, hence its universal covering space is a Hadamard manifold.
\end{theorem}

There's also a generalization of Bochner's theorem on Killing vector fields.

\begin{theorem}[Bochner's theorem for orbifolds {\cite[Theorem 2.6]{caramello2}}]\label{teo: bochner for orbifolds}
Let $\mathcal{O}$ be a connected, compact, orientable Riemannian orbifold with $\ric_{\mathcal{O}}\leq0$. Then every Killing vector field on $\mathcal{O}$ is parallel, and $\dim\mathrm{Iso}(\mathcal{O})\leq \dim\mathcal{O}_{\mathrm{min}}$. Moreover, if $\ric_{\mathcal{O}}<0$, then $\mathrm{Iso}(\mathcal{O})$ is finite.
\end{theorem}

Now in non-negative curvature, Bochner's technique applied to $1$-forms yields that, for a manifold $M$, if $\ric_{M}\geq0$ then $b_1(M)\leq \dim M$, with equality holding if and only if $M$ is a flat torus. This was generalized for orbifolds by J.~Borzellino:

\begin{theorem}[Bochner's theorem for orbifolds {\cite[Theorem 4]{borzellino5}}]\label{teo: bochner for orbifolds for forms}
Let $\mathcal{O}$ be a connected, compact, orientable, $n$-dimensional Riemannian orbifold with $\ric_{\mathcal{O}}\geq0$. Then $b_1(\mathcal{O})\leq n$, with equality holding if and only if $\mathcal{O}$ is a flat torus.
\end{theorem}

Let $\mathcal{O}$ be a compact, oriented, $n$-dimensional Riemannian orbifold. Since $\Omega(\mathcal{O})\cong \Omega(\mathscr{H}_{\mathcal{O}})$, the Hodge star $\star:\Omega^i(\mathcal{O})\to \Omega^{n-1}(\mathcal{O})$ and the codiferential $\delta:\Omega^{i+1}(\mathcal{O})\to \Omega^i(\mathcal{O})$ are well defined, and one has $\proin{d\omega}{\eta}=\proin{\omega}{\delta\eta}$ with respect to the inner product
$$\proin{\omega}{\eta}=\int_{\mathcal{O}} \omega\wedge\star\eta.$$
Then one can consider the Laplacian $\Delta=d\delta+\delta d:\Omega^i(\mathcal{O})\to \Omega^i(\mathcal{O})$ on $\mathcal{O}$. The space of harmonic $i$-forms on $\mathcal{O}$ is denoted
$$\mathcal{H}^i(\mathcal{O})=\{\omega\in \Omega^i(\mathcal{O})\ |\ \Delta\omega=0\}.$$
There is the following Hodge decomposition theorem for orbifolds (which here we cite in less generality for simplicity).

\begin{theorem}[Hodge decomposition for orbifolds {\cite[Theorem 4.4]{farsi}}\footnote{Alternatively, one can apply \cite[Theorems 7.22 and 7.51]{tondeur} to the foliation of $\mathcal{O}^\Yup_+$ given by the orbits of $\mathrm{SO}(n)$, recall Proposition \ref{proposition: every orbifold is a quotient}.}]
Let $\mathcal{O}$ be a compact, oriented Riemannian orbifold. Then each $\Omega^i(\mathcal{O})$ splits orthogonally as
$$\Omega^i(\mathcal{O})\cong \mathrm{Im}(d)\oplus \mathrm{Im}(\delta) \oplus \mathcal{H}^i(\mathcal{O}).$$
In particular $H_{\mathrm{dR}}^i(\mathcal{O})\cong \mathcal{H}^i(\mathcal{O})$.
\end{theorem}

It is a classical result by Gromov that the sum of all Betti numbers of a negatively curved manifold is linearly bounded by its volume (see \cite[p. 12]{gromov2} or \cite[p. 115]{ballmann}). There is a generalization of this theorem to orbifolds due to I.~Samet:

\begin{theorem}[Gromov's theorem for orbifolds {\cite[Theorem 1.1]{samet}}]
Let $\mathcal{O}$ be a connected, complete, $n$-dimensional Riemannian orbifold satisfying $-1\leq\sec_\mathcal{O}<0$. Then
$$\sum_i b_i\leq C_n\vol(|\mathcal{O}|),$$
where $C_n$ is a constant depending only on $n$.
\end{theorem}

We also mention here the following finiteness result under geometric constraints, by J.~Harvey.

\begin{theorem}[Geometric finiteness theorem for orbifolds {\cite[Theorem 1.1]{harvey}}]
Let $k,D,v\in\mathbb{R}$ be fixed. Then the class of all $n$-dimensional Riemannian orbifolds satisfying $\sec_{\mathcal{O}}\geq k$, $\diam(|\mathcal{O}|)\leq D$ and $\vol(|\mathcal{O}|)\geq v$ has only finitely many members, up to orbifold homeomorphisms.
\end{theorem}

Finally, we introduce the orbifold generalization of the Gauss--Bonnet theorem proved by I.~Satake. Let $\mathcal{O}$ be a $2k$-dimensional, orientable, Riemannian orbifold and consider $E_1,\dots,E_{2k}\in\mathfrak{X}(\mathcal{O})$ which restrict to an orthonormal frame on a chart $(\widetilde{U},H,\phi)$. Then we can write
$$R(X,Y)E_j=\sum_i\Omega_j^i(X,Y)E_i,$$
were $\Omega_j^i\in\Omega^2(\widetilde{U})^H$. We view $\Omega_U=(\Omega_j^i)$ as a $\mathfrak{so}(2k)$-valued $2$-form, called the \textit{curvature form} of $\mathcal{O}$ (on $U$). Recall that the \textit{Pfaffian} of a skew symmetric $2k\times 2k$ matrix $X=(x_j^i)$, given by
$$\mathrm{Pf}(X)=\frac{1}{2^kk!}\sum_{\sigma\in S_{k}}\mathrm{sgn}(\sigma)\prod_{i=1}^k x^{\sigma_{2i-1}}_{\sigma_{2i}},$$
satisfy $\det X=\mathrm{Pf}(X)^2$. Moreover, one has $\mathrm{Pf}(T^{-1}XT)=\det(T)\mathrm{Pf}(X)$. This property ensures that the local forms $\mathrm{Pf}(\Omega_U)$ behave well under changes of charts, since the transition functions of $T\mathcal{O}$ are $\mathrm{SO}(2k)$-valued. Hence we obtain a global form $\mathrm{Pf}(\Omega)\in\Omega^{2k}(\mathcal{O})$, called the \textit{Euler form} of $\mathcal{O}$.

\begin{theorem}[Gauss--Bonnet theorem for orbifolds {\cite[Theorem 2]{satake2}}]\label{teo: gauss bonnet}
Let $\mathcal{O}$ be a $(2k)$-dimensional, compact, oriented Riemannian orbifold. Then
$$\int_{\mathcal{O}}\mathrm{Pf}(\Omega)=(2\pi)^{k}\chiorb(\mathcal{O}).$$
\end{theorem}

Moreover, as in the manifold case, for an odd dimensional compact Riemannian orbifold one has $\chiorb(\mathcal{O})=0$ \cite[Theorem 4]{satake2}.


\begin{thebibliography}{10}
\addcontentsline{toc}{chapter}{Bibliography}

\bibitem{adem} A. Adem, J. Leida, Y. Ruan: \textit{Orbifolds and Stringy Topology}, Tracts in Mathematics \textbf{171}, Cambridge University Press, 2007.


\bibitem{alex} M. Alexandrino, R. Betiol: \textit{Lie Groups and Geometric Aspects of Isometric Actions}, Springer, 2015.



\bibitem{allday} C. Allday, V. Puppe: \textit{Cohomological Methods in Transformation Groups}, Cambridge Studies in Advanced Mathematics \textbf{32}, Cambridge University Press, 1993.



\bibitem{ballmann} W. Ballmann, M. Gromov, V. Schroeder: \textit{Manifolds of nonpositive curvature}, Progress in Mathematics \textbf{61}, Birkhäuser Boston Inc, 1985.

\bibitem{bagaev} A. Bagaev, N. Zhukova: \textit{The isometry groups of Riemannian orbifolds}, Sib. Math. J. \textbf{48}(4) (2007), 579--592.

\bibitem{behrend} K. Behrend, P. Xu: \textit{Differentiable Stacks and Gerbes}, arXiv:math/0605694 [math.DG].

\bibitem{biringer} I. Biringer, N. Lazarovich, A. Leitner: \textit{On the Chabauty space of $\mathrm{PSL}_2(\mathbb{R})$, I: lattices and grafting}, arXiv:2110.14401 [math.GT].

\bibitem{borzellino3} J. Borzellino: \textit{Orbifolds of maximal diameter}, Indiana Univ. Math. J. \textbf{42}(1) (1993), 37--53.

\bibitem{borzellino4} J. Borzellino: \textit{Riemannian Geometry of Orbifolds}, University of California, Los Angeles, 1992.

\bibitem{borzellino5} J. Borzellino: \textit{Orbifolds with lower Ricci curvature bounds}, Proc. Am. Math. Soc. \textbf{125}(10) (1997), 3011--3018.

\bibitem{borzellino6} J. Borzellino, V. Brunsden: \textit{Elementary orbifold differential topology}, Topology Appl. \textbf{159}(17) (2012), 3583--3589.

\bibitem{borzellino2} J. Borzellino, V. Brunsden: \textit{The stratified structure of spaces of smooth orbifold mappings}, Commun. Contemp. Math. \textbf{15}(5) (2013), 1350018.

\bibitem{borzellino7} J. Borzellino, V. Brunsden: \textit{On the notions of suborbifold and orbifold embedding}, Algebr. Geom. Topol. \textbf{15} (2015), 2787--2801, 10.2140/agt.2015.15.2789

\bibitem{borzellino8} J. Borzellino, S. Zhu: \textit{The splitting theorem for orbifolds}, Illinois J. Math. \textbf{38} (1994), 679--691.

\bibitem{boyer} C. Boyer, K. Galicki: \textit{Sasakian Geometry}, Oxford Mathematical Monographs, Oxford University Press, 2008.

\bibitem{bridson} M. Bridson, A. Haefliger: \textit{Metric spaces of non-positive curvature} Springer, 2013.


\bibitem{burago} D. Bugargo, Y. Burago, S. Ivanov: \textit{A course in Metric Geometry}, Graduate Studies in Mathematics \textbf{33}, American Mathematical Society, 2001.



\bibitem{caramello} F. Caramello, D. Töben: \textit{Positively curved Killing foliations via deformations}, Trans. Amer. Math. Soc. \textbf{372} (2019), 8131--8158, https://doi.org/10.1090/tran/7893

\bibitem{caramello2} F. Caramello, D. Töben: \textit{Equivariant basic cohomology under deformations}, Math. Z. \textbf{299} (2021) 2461--2482, https://doi.org/10.1007/s00209-021-02768-w

\bibitem{chen} K. Chen: \textit{On differentiable spaces}, in: \textit{Categories in continuum physics}, Lecture Notes in Mathematics \textbf{1174}, Springer, 1986, 38--42.


\bibitem{chenruan} W. Chen, Y. Ruan: \textit{Orbifold Gromov--Witten theory}, in: \textit{Orbifolds in Mathematics and Physics}, ed. by A. Adem, J. Morava and Y. Ruan, Contemporary Mathematics \textbf{310}, American Mathematical Society, 2002, 25--85.

\bibitem{choi} S. Choi: \textit{Geometric Structures on 2-orbifolds: Exploration of Discrete Symmetry}, MSJ Memoirs \textbf{27}, Mathematical Society of Japan, 2012.

\bibitem{cohn} S. Cohn-Vossen: \textit{Existenz kürzester Wege}, Doklady SSSR \textbf{8} (1935), 339--342.

\bibitem{cooper} D. Cooper, C. Hodgson, S. Kerckhoff: \textit{Three-dimensional orbifolds and cone-manifolds}, MSJ Memoirs \textbf{5},  Mathematical Society of Japan, 2000.

\bibitem{dolgachev} I. Dolgachev: \textit{Weighted projective varieties}, in: \textit{Group Actions and Vector Fields}, ed. by J. Carrel. Lecture Notes in Mathematics \textbf{956}, Springer, 1982, 34--71.

\bibitem{dragomir} G. Dragomir: \textit{Closed geodesics on compact developable orbifolds}, Doctoral Dissertation, McMaster University, 2011.



\bibitem{eisenbud} D. Eisenbud: \textit{Commutative Algebra: with a view toward algebraic geometry}, Springer-Verlag, 2013.

\bibitem{emmrich} C. Emmrich, H. Römer: \textit{Orbifolds as configuration spaces of systems with gauge symmetries}, Commun. Math. Phys. \textbf{129} (1990), 69--94.

\bibitem{farsi} C. Farsi, E. Proctor, C. Seaton: \textit{Approximating orbifold spectra using collapsing connected sums}. J. Geom. Anal. \textbf{31} (2021), 9433--9468, https://doi.org/10.1007/s12220-021-00611-6


\bibitem{frolicher} A. Frölicher: \textit{Smooth structures}, in: \textit{Category Theory}, Lecture Notes in Mathematics \textbf{962}, Springer, 1982, 69--81.

\bibitem{galazgarcia} F. Galaz-García et al: \textit{Torus orbifolds, slice-maximal torus actions and rational ellipticity}, Int. Math. Res. Not. IMRN (2017), rnx064.






\bibitem{gromov2} M. Gromov: \textit{Volume and bounded cohomology}, Inst. Hautes Etudes Sci. Publ. Math. \textbf{56} (1982), 5-–99.


\bibitem{guillemin}  V. Guillemin, S. Sternberg: \textit{Supersymmetry and Equivariant de Rham Theory}, Springer-Verlag, 1999.

\bibitem{haefliger3} A. Haefliger: \textit{Groupoïdes d'holonomie et classifiants}, Astérisque \textbf{116} (1986), 321--331.

\bibitem{haefliger4} A. Haefliger, Q. Du: \textit{Une présentation du groupe fondamental d'une orbifold}, Astérisque \textbf{177-178} (1989), 183--197.


\bibitem{haefliger2} A. Haefliger, E. Salem: \textit{Riemannian foliations on simply-connected manifolds and actions of tori on orbifolds}, Illinois J. Math. \textbf{34} (1990), 706--730.


\bibitem{harvey} J. Harvey: \textit{Equivariant Alexandrov geometry and orbifold finiteness}, J. Geom. Anal. \textbf{26} (2016), 1925--1945, https://doi.org/10.1007/s12220-015-9614-6


\bibitem{husemoller} D. Husemoller: \textit{Fibre Bundles}, Third Edition, Graduate Texts in Mathematics \textbf{20}, Springe-Verlag, 1994.


\bibitem{iglesias} P. Iglesias-Zemmour: \textit{Diffeology}, Mathematical Surveys and Monographs, American Mathematical Society, 2013.

\bibitem{kleiner} B. Kleiner, J. Lott: \textit{Geometrization of three-dimensional orbifolds via Ricci flow}, Astérisque \textbf{365} (2014), 101--177.


\bibitem{lange} C. Lange: \textit{Orbifolds from the metric viewpoint}, preprint arXiv:1801.03472v2 [math.DG].

\bibitem{lerman} E. Lerman: \textit{Orbifolds as stacks?}, L'Enseignement Mathématique \textbf{56} (2010), 315--363.




\bibitem{loh} C. Löh: \textit{Geometric group theory}, Universitext, Springer, 2017.

\bibitem{lupercio} E. Lupercio, B. Uribe: \textit{Gerbes over orbifolds and twisted $K$-theory}, Comm. Math. Phys. \textbf{245}(3) (2004), 449--489.


\bibitem{meinrenken} E. Meinrenken: \textit{Symplectic surgery and the spin$^c$-Dirac operator}, Adv. Math. \textbf{134}(2) (1998), 240--277.

\bibitem{meinrenken2} E. Meinrenken: \textit{Equivariant cohomology and the Cartan model}, in: Encyclopedia of mathematical physics. Ed. by J.-P. Françoise, G. Naber and T. Tsun. Elsevier, 2006, 242--250.

\bibitem{mrcun} I. Moerdijk, J. Mr\v{c}un: \textit{Introduction to Foliations and Lie Groupoids}, Cambridge Studies in Advanced Mathematics \textbf{91}, Cambridge University Press, 2003.

\bibitem{moerdijk} I. Moerdijk: \textit{Orbifolds as groupoids: an introduction}, in: Orbifolds in Mathematics and Physics. Ed. by A. Adem, J. Morava and Y. Ruan. Contemporary Mathematics \textbf{310}, American Mathematical Society, 2002, 205--222.

\bibitem{moerdijk3} I Moerdijk, D. Pronk: \textit{Simplicial cohomology of orbifolds}, Indag. Math. \textbf{10}(2) (1999), 269--293, https://doi.org/10.1016/S0019-3577(99)80021-4.


\bibitem{pronk} I. Moerdijk, D. Pronk: \textit{Orbifolds, sheaves and groupoids}, $K$-Theory \textbf{12}(1) (1997), 3--21.

\bibitem{puttmann} T. Püttmann, C. Searle: \textit{The Hopf conjecture for manifolds with low cohomogeneity ar high symmetry rank}, P. Am. Math. Soc. \textbf{130}(1) (2001), 163--166.



\bibitem{milnor} J. Milnor: \textit{A note on curvature and fundamental group}, J. Differ. Geom \textbf{2}(1) (1968), 1--7.






\bibitem{salem} E. Salem: \textit{Riemannian foliations and pseudogroups of isometries}, Appendix D in: \textit{Riemmanian Foliations} by P. Molino, Progress in Mathematics \textbf{73}, Birkhäuser, 1988, 265--296.

\bibitem{samet} I. Samet: \textit{Betti numbers of finite volume orbifolds}, Geom. Topol. \textbf{17} (2013), 1113-1147.

\bibitem{satake} I. Satake: \textit{On a generalization of the notion of manifold}, Proc. Natl. Acad. Sci. USA,  \textbf{42}(6) (1956), 359--363.

\bibitem{satake2} I. Satake: \textit{The Gauss--Bonnet theorem for $V$-manifolds}, J. Math. Soc. Jpn.,  \textbf{9}(4) (1957), 464--492.

\bibitem{schmeding} A. Schmeding: \textit{The diffeomorphism group of a non-compact orbifold}, Diss. Math.,  \textbf{507} (2015), 1--179.

\bibitem{stacey} A. Stacey: \textit{Comparative smootheology}, Theory and Applications of Categories \textbf{25}(4) (2011), 64--117.




\bibitem{tondeur} Ph. Tondeur, {Geometry of Foliations}, Monographs in Mathematics {90}, Birkhäuser, 1997.

\bibitem{thurston} W. Thurston: \textit{The Geometry and Topology of Three-Manifolds}, electronic version 1.1 of lecture notes (2002), www.msri.org/publications/books/gt3m.

\bibitem{yeroshkin} D. Yeroshkin: \textit{Riemannian orbifolds with non-negative curvature}, Doctoral Dissertation, University of Pennsylvania, 2014.

\end{thebibliography}
\end{document}